\documentclass{siamart1116}

\usepackage{lipsum}
\usepackage{amsfonts}
\usepackage{graphicx}
\ifpdf
  \DeclareGraphicsExtensions{.eps,.pdf,.png,.jpg}
\else
  \DeclareGraphicsExtensions{.eps}
\fi

\numberwithin{theorem}{section}

\newtheorem{example}{Example}[section]
\newtheorem{assumption}{Assumption}[section]
\newtheorem{definition}{Definition}[section]
\newtheorem{proposition}{Proposition}[section]
\newtheorem{lemma}{Lemma}[section]
\newtheorem{remark}{Remark}[section]

\newcommand{\TheTitle}{Solvability of the non-linear Dirichlet problem with Integro-Differential operators} 
\newcommand{\RunningTitle}{Dirichlet problem}
\newcommand{\TheAuthors}{E. Bayraktar and Q. Song}

\headers{\RunningTitle}{\TheAuthors}

\title{{\TheTitle}
\thanks{Submitted to the editors in May 2017.}
}

\author{
  Erhan Bayraktar\thanks{University of Michigan
    (\email{erhan@umich.edu}). Erhan Bayraktar is partially supported by the National Science Foundation (DMS-1613170) and the Susan M. Smith Professorship.}
  \and
  Qingshuo Song\thanks{City University of Hong Kong (\email{song.qingshuo@cityu.edu.hk}).}
}

\usepackage{amsopn}

\ifpdf
\hypersetup{
  pdftitle={\TheTitle},
  pdfauthor={\TheAuthors}
}
\fi

\numberwithin{equation}{section}

\begin{document}

\maketitle

\begin{abstract}
  This paper analyzes the solvability of a class of elliptic non-linear Dirichlet problems with jumps. 
The contribution of the paper is the construction of the supersolution required in Perron's method. This is achieved by solving the exit time problem of an It\^{o} jump diffusion. The proof of this relies on the proof of continuity of the entrance time and point with respect to the 
 Skorohod topology. 
\end{abstract}



\section{Introduction}
\subsection*{Problem setup}
Consider an equation of the form
\begin{equation}
 \label{eq:pde11}
   F (u,  x)  + u(x) - \ell(x)
 = 0, \ x \in O
\end{equation}
with the boundary value 
\begin{equation}
 \label{eq:bd11}
 u(x) = g(x), \ x\in O^{c}.
\end{equation}
Here
$$F(u,x) = - \inf_{a\in [\underline a, \overline a]} H(u,x,a) - \mathcal I(u,x) $$ 
where  $\underline a \le \overline a$ are given two real numbers and
\begin{equation}
 \label{eq:I01}
  \mathcal I(u,x) = \int_{\mathbb R^{d}} (u(x+y) - u(x) - D u(x) \cdot y I_{B_{1}}(y)) \nu(dy),
\end{equation}
$$H(u,x,a) = \frac 1 2 tr(A(a) D^{2} u (x) ) + b(a) \cdot Du (x),$$
 with $A(a) = \sigma'(a) \sigma(a)$, and $\nu(\cdot)$ is a L\'{e}vy measure on $\mathbb R^{d}$, i.e. 
 $\int_{\mathbb R^{d}} (1 \wedge |y|^{2}) \nu(dy) <\infty$. Here,
$B_{r}(x)$ is a ball of radius $r$ with center $x$, and
we denote $B_{r}(0)$ by $B_{r}$ for simplicity.
To simplify our presentation, we will use the following additional set of assumptions throughout the paper.
\begin{assumption}
 \label{a:01v5}
\begin{enumerate}
 \item $O$ is a connected open bounded set in $\mathbb R^{d}$.
 \item \label{a:011v5}$\sigma, b\in C^{0,1}(\mathbb R)$; $\ell, g\in C_{0}(\mathbb R^{d})$.
 \item \label{a:012v5}  $\nu(dy) =  \hat \nu(y) dy$ is a L\'{e}vy measure satisfying $\hat \nu \in C_{0}(\mathbb R^{d} \setminus \{0\})$.
\end{enumerate}
\end{assumption}
For some $\alpha \in (0,2)$,  if $\nu$ is given by
$$\nu (dy) = \frac{dy}{|y|^{d+\alpha}},$$ then $\nu$ satisfies Assumption \ref{a:01v5}, and the integral operator is denoted by 
 $\mathcal I (u,x) = - (-\Delta)^{\alpha/2} u(x)$ as convention. For convenience, we write 
 $- (-\Delta)^{0} u = 0$.

\subsection*{Literature review and a motivating example}
A function $u$ is said to be a solution of  
Dirichlet problem \eqref{eq:pde11}-\eqref{eq:bd11}, 
if $u\in C(\bar O)$ 
satisfies \eqref{eq:pde11} in the viscosity sense in $O$ 
and $u = g$ on $O^{c}$.
It is worth to note that, as far as Dirichlet problem \eqref{eq:pde11}-\eqref{eq:bd11} concerned, one can generalize the boundary condition
\eqref{eq:bd11} by 
\begin{equation}
 \label{eq:bd12}
 \max\{F (u,  x)  + u(x) - \ell(x), u - g\} \ge 0 \ge \min\{F (u,  x)  + u(x) - \ell(x), u - g\} \hbox{ on } O^{c}
\end{equation}
without loss of uniqueness in the viscosity sense. 

In contrast to the (classical) 
Dirichlet problem \eqref{eq:pde11}-\eqref{eq:bd11},
Dirichlet problem \eqref{eq:pde11}-\eqref{eq:bd12} is referred to a
generalized Dirichlet problem. 
For the generalized Dirichlet problem without nonlocal operator, there were
many excellent discussions on  the solvability with the
comparison principle and Perron's method, see for instance, 
\cite{BP88}, \cite{BP90}, \cite{BB95}, and 
Section 7 of \cite{CIL92}.
Also see  \cite{FS06} and \cite{MR2763549} for an analysis of this using the dynamic programming principle. Recently, the solvability result 
has been extended to nonlinear equations associated to 
Integro-differential operators, 
see \cite{BI08}, \cite{BCI08}, \cite{AT96},
\cite{Top14}, and the references therein.

Compared to the generalized Dirichlet problem, there are relatively less 
discussions available on the classical Dirichlet problem associated with 
the Integral operators in the aforementioned references. The following example motivates our analysis:
\begin{example}\label{e:11}
Determine the existence and uniqueness of the viscosity solution for the Dirichlet problem  given by, 
\begin{equation}
 \label{eq:toy1}
 | \partial_{x_{1}} u | + (-\Delta)^{\alpha/2} u + u - 1 = 0, \ \forall x\in O = (-1,1) \times (-1, 1),
\end{equation}
 where $\alpha \in [0,2]$, with the boundary condition
 $$ u(x) = 0, \ \forall x\in O^{c}.$$
 \end{example}

This problem is only partially resolved  in the existing literature:
\begin{itemize}
 \item If $\alpha = 0$, there is no solution. In fact, one can directly check that
 $u(x) = 1 - e^{-1+|x_{1}|}$ is the unique solution of the generalized Dirichlet problem, but not a solution of classical 
 Dirichlet problem due to its loss of boundary
 at $\{(x_{1}, x_{2}): |x_{2}| = 1, |x_{1}| < 1\}$.
 \item If $\alpha \in [1,2]$, there is a unique solution by \cite{BCI08}.
 \item If $\alpha \in (0,1)$, although there is unique solution of generalized Dirichlet problem by \cite{Top14}, it was not known 
 whether this solution solves the classical Dirichlet problem. Our main result Theorem~\ref{t:05v5} demonstrates that this in fact is the case, see Example \ref{e:11a}. It is also pointed out there that existence and uniqueness still 
holds for all $\alpha \in (0, 2]$
as long as the boundary satisfies exterior cone condition, which itself is a new result.
\end{itemize}

\subsection{Work outline}
This work focuses on the sufficient condition of 
the existence and uniqueness of the viscosity solution for  
Dirichlet problem of \eqref{eq:pde11}-\eqref{eq:bd11}. 

One alternative in proving this result is using the stochastic Perron methodology introduced by \cite{MR2929032, MR3162260, BS13}, and \cite{MR3295681} for the application of this approach to a particular exit time problem. With this methodology one can in fact identify the value function of the exit time control problem with the generalized Dirichlet problem \eqref{eq:pde11}-\eqref{eq:bd12} using a similar analysis to the proofs of Theorems 2 and 3 in \cite{MR3217159}. Then as in \cite{FS06} (also see \cite{MR2763549}), if we can a priori show that the value function is continuous (this can fail at the boundary), we can conclude that the value function also solves the classical Dirichlet problem \eqref{eq:pde11}-\eqref{eq:bd11}.

Since either one needs to prove continuity separately or has to impose a stronger version of the comparison principle as in Theorem 1 of \cite{MR3217159}, we will not pursue the stochastic Perron approach here. We will instead approach this problem using the classical Perron method. Using the idea of constructing a supersolution satisfying the boundary conditions from an auxiliary stochastic exit time  problem as in \cite{MR3295680} (and in \cite{MR3295681} in a slightly different set-up), we will be able to apply \cite{BI08} and obtain a unique viscosity solution. This result, which is the main contribution of the paper, is presented in Theorem \ref{t:05v5}. 

The technical step of the proof of Theorem \ref{t:05v5} involves proving the continuity of the value function of the exit time problem of an It\^{o} jump diffusion, see Proposition~\ref{p:03v5}. 
In general, due to the non-local property, continuity of the value function
up to a stopping time is much more delicate than the 
counterpart of the purely differential form.
We establish this result by investigating the continuity set of the 
of entrance time and entrance point mappings with respect to the Skorohod topology; see Theorems~\ref{t:hit01} and \ref{t:hit02}. Then we show that these sets have full measure under our assumption in the proof of Proposition~\ref{p:03v5}.
It is easy to show that
the continuous sample paths are a subset of the points of lower semi-continuity of the entrance time to a closed interval, see e.g. \cite{MR2846253}. 
However, the continuity set is difficult to identify. In fact, continuity does not hold in general for the entrance time as shown in Appendix \ref{sec:skorohod} (see Example~\ref{e:01}) or Page 657 of \cite{SV72}.  Moreover,
Example~\ref{e:02v5} 
demonstrates that the situation for the continuity of the entrance point mapping is even worse.
Our contribution here is the identification of the discontinuity set as a null set under our assumption about the geometry of the boundary. 

\section{Existence of a unique 
solution for the Dirichlet problem} 


\subsection{Two different definitions of viscosity properties}
In this section, we give two different definitions 
of viscosity properties, Definition \ref{d:01v5} and Definition \ref{d:11} 
respectively. Definition \ref{d:01v5} involves only with $C^{2}$ smooth test functions, which will be used later 
to verify the supersolution property of 
a certain value function associated to some exit control problem. 
Compared to Definition \ref{d:01v5}, Definition \ref{d:11} is given with more test functions including non-smooth functions, and it's much harder to 
be used directly in this paper to verify viscosity solution property.
However, Definition \ref{d:11} of this paper 
is exactly  Definition 2 of \cite{BI08}, 
where  it was used to provide 
the proof of comparison principle and Perron's method.
In this connection, we shall prove  
the equivalence of Definition \ref{d:11} and Definition \ref{d:01v5}.

Definition \ref{d:01v5} below is consistent to the Definition 1 of \cite{BI08},
which will be used to establish the existence of the solution in this paper.
To proceed, for a function $u: \bar O \mapsto \mathbb R$, 
we define its extension by
$$u^{g} = (u I_{\bar O} + g I_{\bar O^{c}})^{*}, \quad 
u_{g} =  (u I_{\bar O} + g I_{\bar O^{c}})_{*},$$
where $f^{*}$ and $f_{*}$ stand for USC (upper semicontinuous) 
and LSC (Lower semicontinuous) envelopes 
of the function $f$, respectively. We also
define the supertest function space, 
for $u \in USC$ and 
$x\in \mathbb R^{d}$
\begin{equation}\label{eq:J+}
 J^{+} (u, x) = \{\phi \in C_{b}^{\infty}(\mathbb R^{d}), 
\hbox{ such that } \phi \ge u^{g} \hbox{ and } \phi(x) = u(x)\}.
\end{equation}
Analogously, the subtest function space is given by, for $u\in LSC$ 
and  $x\in \mathbb R^{d}$ 
\begin{equation}
 \label{eq:J-}
 J^{-} (u, x) = \{\phi \in C_{b}^{\infty}(\mathbb R^{d}), 
\hbox{ such that } \phi \le u_{g} \hbox{ and } \phi(x) = u(x)\}.
\end{equation}

\begin{definition}
 \label{d:01v5}
\begin{enumerate}
 \item   We say  a function $u \in USC(\bar O)$ satisfies the viscosity 
 subsolution property at $x\in  O$, 
 if the following inequality holds,
\begin{equation}
 \label{eq:sub11}
F (\phi,  x)  + u(x) - \ell(x)
 \le 0,  \ \forall \phi \in J^{+} (u,x).
\end{equation}
\item   We say  a function $u \in LSC(\bar O)$
 satisfies the viscosity 
 supersolution property at $x\in  O$, 
 if the following inequality holds,
\begin{equation}
 \label{eq:sup11}
F (\phi,  x)  + u(x) - \ell(x)
 \ge 0, \ \forall \phi \in J^{-} (u,x).
\end{equation}
\end{enumerate}
\end{definition}

Next, we observe that $\phi \mapsto F (\phi,  x)$ of \eqref{eq:sub11} and
\eqref{eq:sup11} could be well defined for a function being $C^{\infty}$-smooth only at some neighborhood of $x$. Indeed, 
for an arbitrary $x\in \mathbb R^{d}$, if we define a function space 
$C_{x}$ by
\begin{equation}
 \label{eq:phiv4}
 C_{x} = \{\phi : \exists \hat r>0, \
 \phi_{1} \in C^{\infty}, \ \phi_{2} \in L^{1}, \hbox{ s.t. } 
 \phi = \phi_{1} I_{\bar B_{\hat r}(x)} + \phi_{2} (1 - I_{\bar B_{\hat r}(x)})\},
\end{equation}
one can directly verify that 
$\phi \mapsto \mathcal I(\phi, x)$ is well defined for $\phi \in C_{x}$, with
a property 
\begin{equation}
 \label{eq:03v5}
 \mathcal I(\phi, x) = b_{r} \cdot D \phi(x) + 
\mathcal I_{r,1} (\phi,x) + \mathcal I_{r,2}(\phi,x),  \ \forall r>0
\end{equation}
where 
\begin{enumerate}
 \item $b_{r} = \int_{B_{1}\setminus B_{r}} y \nu(dy).$
 \item  $\mathcal I_{r,1}(\phi, x) = \int_{B_{r}} (\phi(x+y) - \phi(x) - D \phi(x) \cdot y) \nu(dy)$
 \item 
 $\mathcal I_{r,2} (\phi, x) = \int_{\mathbb R \setminus B_{r}} 
(\phi(x+y) - \phi(x)) \nu(dy)$.
\end{enumerate}
In the above, $\int_{B_{1}\setminus B_{r}}$ for $r>1$ 
is understood as $-\int_{B_{r}\setminus B_{1}}$.
Note that, 
(a) the identity \eqref{eq:03v5} agrees with the original definition \eqref{eq:I01} of $\mathcal I$; (b) 
$r$ in \eqref{eq:03v5} could be larger than $\hat r$ of  \eqref{eq:phiv4}.
This observation allows us to use more 
test functions from $C^{\infty}$ to $C_{x}$ compared to 
Definition \ref{d:01v5}.
In this below, 
Definition \ref{d:11} is consistent to  Definition 2 of \cite{BI08}.

\begin{definition}
 \label{d:11}
 \begin{enumerate}
 \item  We say  a function $u \in USC(\bar O)$ satisfies the viscosity 
 subsolution property at $x\in  O$, if for all  $\phi \in C_{x}$  with
 (1)
 $\phi(x) = u(x)$; (2) $\phi - u \ge 0$ on $\bar O$, satisfies
\begin{equation}
 \label{eq:sub111}
 - b_{r} \cdot D \phi(x) - \mathcal I_{r,1} (\phi,x) - \mathcal I_{r,2} (u^{g},x) 
 - \inf_{a\in [\underline a, \overline a]} H(\phi,x,a) + u(x) - \ell(x) \le 0, \ \forall r>0,
\end{equation}
 \item  We say  a function $u \in LSC(\bar O)$ satisfies the viscosity 
 supersolution property at $x\in  O$, if for all  $\phi \in C_{x}$  with
 (1)
 $\phi(x) = u(x)$; (2) $\phi - u \le 0$ on $\bar O$
\begin{equation}
 \label{eq:sup111}
  - b_{r} \cdot D \phi(x) - \mathcal I_{r,1} (\phi,x) - \mathcal I_{r,2} (u_{g},x) 
 - \inf_{a\in [\underline a, \overline a]} H(\phi,x,a) + u(x) - \ell(x) \ge 0, \ \forall r>0.
\end{equation}
\end{enumerate}
\end{definition}

\begin{proposition}
 \label{p:01v6}
 Definition \ref{d:01v5} is equivalent to Definition \ref{d:11}.
\end{proposition}
The proof is relegated to Appendix ~\ref{s:def}.

\subsection{Perron's method} \label{sec:perron}

\begin{definition} \label{d:03}
A function $u\in USC(\bar O)$ {\rm (resp. $u\in LSC(\bar O)$)} is said to be a viscosity subsolution {\rm (resp. supersolution)}  of 
\eqref{eq:pde11} - \eqref{eq:bd11}, if
 $u$ satisfies the  subsolution {\rm (resp. supersolution) } property 
 at each $x\in O$ and $u = g$ at $O^{c}$. A function $u\in C(\bar O)$ is said to be a solution of \eqref{eq:pde11} - \eqref{eq:bd11}, if it is a sub and supersolution of \eqref{eq:pde11} - \eqref{eq:bd11} at the same time.
\end{definition}

\begin{proposition}  [Comparison Principle]
 \label{p:cp}
 If $u$ and $v$ are subsolution and supersolution of \eqref{eq:pde11} - \eqref{eq:bd11}, then $u\le v$.
\end{proposition}
\begin{proof}
 Since
Definition \ref{d:01v5} is in fact equivalent to Definition 2 of \cite{BI08} by 
Proposition~\ref{p:01v6}, the statement above follows from the corresponding statement in Theorem 3 of \cite{BI08}.
\end{proof}

\begin{proposition} [Perron's Method]
 \label{p:pm}
If there exist a subsolution $\underline u$ and a supersolution $\overline u$ to  \eqref{eq:pde11} - \eqref{eq:bd11}, then 
 $$w(x) = \inf\{u \in LSC(\bar O) : u \hbox{ is subsolution}\}$$
 is the unique solution in $C(\bar O)$. 
\end{proposition}

We relegate the proof of Proposition~\ref{p:pm} 
in Appendix~\ref{sec:per}.

\begin{remark}
According to Propositions~\ref{p:cp} and  \ref{p:pm}, 
the remaining task is to show the existence of a subsolution $\underline u$ and  a supersolution $\overline u$. 
In general, 
as far as the classical  Dirichlet boundary concerned, 
one shall not expect the existence of subsolution and
supersolution for free due to Example 7.8 of \cite{CIL92}. In this regard, 
some sufficient conditions of the existence of subsolution and
supersolution of Dirichlet problem is provided by Example 4.6 of \cite{CIL92}, and the general case has remained open. In this paper, we address the issue of constructing a supersolution, which we carry out in the next subsection. 
\end{remark}

\subsection{Stochastic exit control problem for an It\^{o} jump diffusion}
\label{sec:exit}
To proceed, we consider an exit control problem with Markovian policy.
We consider a fixed filtered probability space $(\Omega, \mathcal F, \mathbb P, \{\mathcal F_{t}, t>0\})$, on which $W$ is a standard Brownian motion and $L$ is a L\'{e}vy process with generating triplet $(0, \nu, 0)$,
see notions of L\'{e}vy process in \cite{Sat99} or \cite{Ber96}.
We consider a stochastic
differential equation controlled by a Lipschitz continuous
function $m: \mathbb R^{d} \mapsto [\underline a, \overline a]$, 
\begin{equation}
 \label{eq:sdev5}
 X_{t} = x + \int_{0}^{t} b(m(X_{s})) dt + \int_{0}^{t} \sigma(m(X_{s})) dW_{s} + L_{t},
\end{equation}
By \cite{App04}, \eqref{eq:sdev5} admits a unique solution which has 
a c\`adl\`ag version, and we assume $X$ to be a c\`adl\`ag process.
Next, we define the first exit times
\begin{equation}
 \label{eq:tau}
 \tau = \inf\{t>0, X_{t} \notin O\}
\end{equation}
and 
\begin{equation}
 \label{eq:tauhat}
 \hat \tau = \inf\{t>0, X_{t} \notin \bar O\}.
\end{equation}
Let $\mathbb D^{d}_{\infty}$ be
the space of  c\`adl\`ag functions on $[0, \infty)$
with Skorohod metric given by $d^{o}_{\infty}$,
see detailed definition in Section~\ref{sec:skorohod}.
We are interested in the following subset of Markovian policy space $\mathcal M$ defined  by
\begin{equation} \label{eq:05v5}
 \mathcal M = \{m \in C^{0,1}(\mathbb R^{d}, [\underline a, \overline a]) : \mathbb P^{m,x} (\hat \tau = 0) = 1, 
 \ \forall x\in \partial O\}.
\end{equation}
For a given $(x,m) \in \mathbb R^{d} \times \mathcal M$, we use
$\mathbb P^{m,x}$ to denote the probability measure 
on $\mathbb D^{d}_{\infty}$  induced by $X_{t}$, i.e.
$\mathbb P^{m,x} (B) = \mathbb P(X\in B)$ for all Borel set $B$ of
$(\mathbb D^{d}_{\infty}, d^{o}_{\infty})$ .
We also use $\mathbb E^{m,x}$ to denote the expectation operator with respect to $\mathbb P^{m,x}$.
\begin{proposition}
 \label{p:03v5}
 Let $m\in \mathcal M$ of \eqref{eq:05v5}, and
 $$V_{m}(x) := \mathbb E^{m,x} \big[ \int_{0}^{\tau} e^{-s} \ell(X_{s}) ds
+ g(X_{\tau})\big]$$
with $\tau$ given by \eqref{eq:tau}. Then, the function $V_{m}$ belongs to 
 $C(\bar O)$.
\end{proposition}
\begin{proof}
This result is a corollary of the technical results presented in Theorems~\ref{t:hit01} and \ref{t:hit02}. See Section~\ref{prop:PofProVm}.
\end{proof}

\subsection{Main result}
\label{sec:exist}

We next state the main result of this paper, which is a corollary of Proposition~\ref{p:03v5}.

\begin{theorem}
 \label{t:05v5}
 If $\mathcal M \neq \emptyset$ and $u = g$ is a subsolution of  \eqref{eq:pde11} - \eqref{eq:bd11},  then there exists a unique 
 continuous viscosity solution of  \eqref{eq:pde11} - \eqref{eq:bd11}.
\end{theorem}

\begin{remark}
The sufficient condition in Theorem \ref{t:05v5} 
requires some regularity of the boundary with respect to 
some controlled process, this requirement is not that strong. For example, the regularity in Example 4.6 of \cite{CIL92} and \cite{BB95} asks the boundary to be $C^{2}$. We can in fact consider non-smooth boundaries
satisfying exterior cone condition
with some appropriate Integro-differential operators,
see the first paragraph of Example \ref{e:11a}, for instance.
\end{remark}

\emph{Proof of Theorem~\ref{t:05v5}}.
The uniqueness holds by Proposition~\ref{p:cp}, and we shall prove
the existence by Perron's method Proposition~\ref{p:pm}. To proceed, we shall find out 
sub and supersolution. Note that $g$ is a subsolution and below we will show that $V_{m}$ is a supersolution for any $m\in \mathcal M$.

We fix a policy $m\in \mathcal M$. By Proposition \ref{p:03v5}, we have $V_{m} \in C(\bar O)$ with $V_{m}(x) = g(x)$ for all $x\in \partial O$. So, it's enough to show that $V_{m}$ satisfies the supersolution property in $O$, i.e.
\begin{equation*}
 \label{eq:exi01}
 F_{m} (\phi, x) + \phi(x) - \ell(x) \ge 0, \ \forall x\in O, \phi \in 
 J^{-}(V_{m}, x).
\end{equation*}
where $F_{m}(\phi, x) = - H(\phi, x, m(x)) - \mathcal I(\phi, x)$.
To the contrary, let's assume 
$$F_{m} (\phi, x) + \phi(x) - \ell(x) = - \epsilon <0$$ for some $x\in O$ and 
$\phi \in J^{-}(V_{m}, x)$.
By Lemma \ref{p:02v5} and the continuity of $m$, the function $ F_{m} (\phi, \cdot)$ is continuous at $x$, and there exists $h>0$ that 
\begin{equation}
 \label{eq:exi21}
 \sup_{|y-x|< h} F_{m} (\phi,  y) + \phi(y) - \ell(y) < - \epsilon/2.
\end{equation}
Since $X$ of \eqref{eq:sdev5} is a c\`adl\`ag process, the first exit time
satisfies $\mathbb P^{m,x}\{\tau >0 \} = 1$.
By the strong Markov property of the process
$X$, we rewrite the value function $V_{m}$ 
as, for any stopping time $\theta \in (0, \tau]$
\begin{equation*} \label{eq:dpp11}
 V_{m}(x) = \mathbb E^{m,x} \Big[ e^{- \theta} V_{m} (X_{\theta}) + \int_{0}^{\theta} e^{-s}\ell(X_{s})ds \Big],
\end{equation*}
which in turn implies  that, with the fact of $\phi \in J^{-} (V_{m}, x)$
\begin{equation*}
 \phi(x) \ge \mathbb E^{m,x} \Big[ e^{- \theta} \phi(X_{\theta}) + \int_{0}^{\theta} e^{-s}\ell(X_{s})ds \Big],
\end{equation*}
On the other hand, one can use Dynkin's formula on $\phi$ to write
$$
\mathbb E^{m, x} [e^{- \theta} \phi(X_{\theta})] = \phi(x) -  
\mathbb E^{m, x} 
\Big [\int_{0}^{\theta} e^{- s} (F_{m} ( \phi, X_{s}) + \phi(X_{s})) ds \Big].
$$
By adding up the above two formulas together, it yields that
$$
\mathbb E^{m, x} 
\Big [\int_{0}^{\theta} e^{- s} (F_{m} ( \phi, X_{s}) + \phi(X_{s}) - \ell(X_{s})) ds \Big] \ge 0.
$$
Finally we take $\theta = \inf\{t >0: X(t)  \notin B_{h}(x)\} \wedge \tau$ in the above and note that 
$\theta > 0$ almost surely in $\mathbb P^{m, x}$. This leads to 
a  contradiction to  \eqref{eq:exi21}. 
\hfill $\square$

\begin{remark}
The sufficient condition of Theorem \ref{t:05v5} consists of (1) $\mathcal M \neq \emptyset$; and (2) subsolution property $g$ to ensure the uniqueness and existence of the solution. 
We will give two examples. In the first example we will address the open problem we posed in Example~~\ref{e:11} (the condition that $\mathcal M \neq \emptyset$ is satisfied). In the second 
example, we will address the necessity of the assumption on $g$.
\end{remark}

\begin{example}[Resolution of the open problem in Example~~\ref{e:11}]
 \label{e:11a}
 {\rm
Consider the set-up in Example~\ref{e:11} with $\alpha \in (0,2)$. We address the existence and uniqueness problem we proposed below. 
We should point out that our proof would not be affected 
 if
 the domain $O$ is 
 replaced by any open connected set satisfying exterior 
 cone condition.
 
 We first rewrite the equation \eqref{eq:toy1} as
$$- \inf_{a \in [-1,1]} \{a  \ \partial_{x_{1}} u \} + (- \Delta)^{\alpha/2} u + u - 1 = 0 \ \hbox{ on } O.$$
For $m\in \mathcal M$, we set
$$X_{t} = x + \int_{0}^{t} m(X_{s}) e_{1} ds + L_{t}^{\alpha}$$
where $e_{1} = (1, 0)'$ is a unit vector and $L^{\alpha}$ is a symmetric 
$\alpha$-stable process with the generating triplet 
$(0, \nu (dy) = \frac{dy}{|y|^{d+\alpha}}, 0)$.
The corresponding value function is
$$V_{m}(x) = \mathbb E^{m,x} \Big[ \int_{0}^{\tau} e^{-s} ds\Big] = \mathbb E^{m,x} [1 - e^{-\tau}]$$
with the first exit time $\tau = \inf\{t>0, X_{t} \notin O\}$. One can directly
check both conditions required by Theorem \ref{t:05v5}:
\begin{itemize}
 \item If $\alpha>0$, then we take $m(x) = 0$ and corresponding $X$ is given by
 $$X_{t} = x + L_{t}^{\alpha};$$
 In this case,  $\mathbb P^{m,x} \{\hat \tau = 0\} = 1$ for all $x\in \partial O$ and $\mathcal M \neq \emptyset$.
 \item $ u = 0$ is subsolution. 
\end{itemize}
\hfill $\square$
}
\end{example}

\begin{example} [On the necessity of the subsolution property of $g$]
 \label{e:11b}{\rm 
In terms of subsolution property of $g$ in Dirichlet problem, the boundary data $g$ shall be understood as any USC function $\bar g$ with 
$\bar g = g$ outside of the domain. This condition indeed a relaxation of the
condition V.2.11 of \cite{FS06}. 

 One can check $u(x) = 1 - e^{-1+|x|}$ is the unique solution of
$$| u' | + u - 1 = 0, \ \forall x\in (-1,1) 
\hbox{ with } u (\pm 1) = 0.$$
However, there is no solution for 
$$| u' | + u + 1 = 0, \ \forall x\in (-1,1) \hbox{ with } u (\pm 1) = 0.$$
Indeed, if there were a solution $u$, the boundary condition $u(1) = 0$
implies that $| u' | + u + 1 > 1/2$ in some neighborhood of $1$
due to the continuity of $u$, which leads to a contradiction. 
One can see that this equation does not satisfy the second condition, i.e. $u= 0$ is not subsolution.
\hfill $\square$
} 
\end{example}


\section{Continuity of Entrance time and point}\label{s:exit1}

In this section we will prove Proposition~\ref{p:03v5}, which is the main ingredient of our main result. This result itself depends on two technical results, Theorems~\ref{t:hit01} and \ref{t:hit02}, which we consider as important technical contributions. First we will introduce some notations to state these results and motivate them. The proofs of these two results are lengthy. Therefore, after stating these results, we will first prove Proposition~\ref{p:03v5} (see Section~\ref{prop:PofProVm}) as a corollary and then return to proving Theorems~\ref{t:hit01} and \ref{t:hit02} in Section~\ref{sec:secmnrs}.

We denote by $(\mathbb D^{d}_{t}, d^{o}_{t})$ the 
complete space of  c\`adl\`ag functions on $[0, t)$ 
taking values in $\mathbb R^{d}$ with Skorohod metric $d^{o}_{t}$, 
and by $(\mathbb D^{d}_{\infty}, d^{o}_{\infty})$ the space of  c\`adl\`ag functions 
on $[0, \infty)$. Since there is variance definitions on the Skorohod metric 
in the literature, we provide the explicit definition of Skorohod metric adopted by this paper in  Appendix \ref{sec:skorohod} taken from \cite{Bil99}.

We also define the entrance time operator
$T_{A}: \mathbb D^{d}_{\infty} \mapsto \mathbb R$ by, 
for a set $A\in \mathbb R^{d}$ and $a\in (0,\infty)$
\begin{equation}
 \label{eq:pi01}
 T_{A}(\omega) = \inf\{t\ge 0: \omega(t) \in A\}, \ 
  T_{A}^{a}(\omega) = \inf\{t\ge 0: \omega(t) \in A\}\wedge a, 
\end{equation}
By convention, $T_{A}(\omega) = \infty$ if $\omega(t) \notin A$ for all $t\ge0$.
Given a set $O$, we will call $T_{O^{c}}(\omega)$ as the exit time of $\omega$
from the set $O$.

As in \cite{Bil99} let $\Pi: \mathbb D^{d}_{\infty} \times [0, \infty) \mapsto \mathbb R^{d}$ defined by $\Pi( \omega,t) = \omega(t)$. 
Similarly, define the value at the first entrance point by
\begin{equation}
 \label{eq:PiO}
\Pi_{O}(\omega) =  \omega(T_{O^{c}}(\omega)).
\end{equation}
Our goal is to investigate the sufficient condition such that the 
mappings 
$T_{O^{c}}$ and $\Pi_{O}$ are continuous for a given set $O$, 
and this will serve as an important tool for the existence of the solution.

\begin{remark}
Example \ref{e:01} shows that $T_{O^{c}}$ is neither 
upper semicontinuous nor lower semicontinuous in general. 
Moreover, Example~\ref{e:02v5} 
demonstrates that the situation for the continuity of 
$\Pi_{O}$ is even worse than the mapping $T_{O^{c}}$.
\end{remark}

The following theorems are the main results of this section 
on the continuity of the two mappings
$T_{O^{c}}$ and $\Pi_{O}$, 
and their proofs will be relegated to Section \ref{sec:pfhit01} and \ref{sec:pfhit02}. Roughly speaking, both $T_{O^{c}}$ and $\Pi_{O}$ are continuous at $\omega$ if, at the first exit time
\begin{enumerate}
 \item either $\omega$ exits from $O$ to $\bar O$ continuously by crossing $\partial O$;
 \item or $\omega$ jumps from a point of $O$ to another point of $\bar O^{c}$.
\end{enumerate}
\begin{theorem}\label{t:hit01}
 $T_{O^{c}} $ is continuous w.r.t. Skorohod metric
 at any $\omega\in \Gamma_{O}$ where
\begin{equation}
 \label{eq:hit04}
 \Gamma_{O} := \{\omega\in \mathbb D^{d}_{\infty}: 
 T_{O^{c}}(\omega^{-}) = T_{O^{c}}(\omega) = T_{\bar O^c}(\omega)\}.
\end{equation}
Here
$$\omega^{-} (t) = \lim_{s\to t-} \omega(s), \ \forall \omega\in \mathbb D^{d}_{\infty}.$$
\end{theorem}
\begin{remark}
It is worth noting that $\Gamma_{O}$ is not a superset of the continuous sample paths, since the second inequality in its definition may not be satisfied. So the lower semi-continuity side of the proof does not follow from the result in \cite{MR2846253}, which shows that the continuous sample paths to be in the points of lower semi-continuity of the above map.
\end{remark}

\begin{theorem}
 \label{t:hit02}
 $\Pi_{O}$  is continuous w.r.t. Skorohod metric at any $\omega\in \hat \Gamma_{O}$
\begin{equation}
 \label{eq:Gammah}
 \hat \Gamma_{O}: = \{\omega \in \Gamma_{O}:
 \hbox{ if }
 \Pi_{O}(\omega^{-}) \in \partial O, \hbox{ then } 
 \Pi_{O}(\omega^{-}) = \Pi_{O}(\omega) \}
\end{equation}
\end{theorem}

\subsection{Proof of Proposition~\ref{p:03v5}}~\label{prop:PofProVm}
Let us denote $b_{m} = b\circ m$ and $\sigma_{m} = \sigma\circ m$. 

 If $x\in \partial O$, then $\tau = 0$ $\mathbb P^{m,x}$-almost surely 
 by definition and
$V_{m}(x) = g(x)$. In the rest of the proof, let $x_{n} \to x \in \bar O$, and we will show the continuity of $V_{m}$ at $x$.

\textbf{Step 1.}
In this step, we will show $\mathbb P^{m, x} (\hat \Gamma_{O} ) =1$
 for all $x\in \bar O$ and $\hat \Gamma_{O}$ defined by 
 \eqref{eq:Gammah}.
 Since both $b_{m}$ and $\sigma_{m}$ are Lipschitz continuous, there exists unique strong solution $X$, which is c\`adl\`ag process with strong Markovian property, see Example 6.4.7 of \cite{App04}. Therefore, $m\in \mathcal M$ implies
\begin{equation}
 \label{eq:06v5}
  \mathbb P^{m,x} \{\tau = \hat \tau\} = 1, \ \forall x\in \bar O.
\end{equation}
 Hence, for all $x\in \partial O$, we have $\Gamma_{O} = \hat \Gamma_{O}$ and $\mathbb P^{m, x} (\hat \Gamma_{O} ) =1$. 
 Now, it remains to show
 $\mathbb P^{m, x} (\hat \Gamma_{O} ) =1, \ \forall x\in O. $
Let $x\in O$ and $\bar \tau = T_{O^{c}}(X^{-})$.
 We define 
 $$\bar \tau_{A} = \left\{
\begin{array}
 {ll}
 \bar \tau & \hbox{ if } \omega \in A;\\
 \infty & \hbox{ otherwise }
\end{array}\right. \hbox{ and } 
\bar \tau_{B} = \left\{
\begin{array}
 {ll}
 \bar \tau & \hbox{ if } \omega \in B;\\
 \infty & \hbox{ otherwise. }
\end{array}\right.
 $$
where 
 $$A = \{X^{-}(\bar \tau) 
 \in \partial O\} \hbox{ and } 
 B = \{X^{-}(\bar \tau) \neq X(\bar \tau)\}.$$

 The left continuity of $X^{-}$ implies $A\in \mathcal F_{\bar \tau-}$ and  the hitting time $\bar \tau_{A}$ is
 a predictable stopping time, while $\bar \tau_{B}$ is totally inaccessible
 stopping time due to the jump by
Meyer's theorem, see Theorem III.4 of \cite{Pro04}. 
Therefore, we conclude 
$\mathbb P^{m,x}(\bar \tau_{A} = \bar \tau_{B}) = 0$ 
by Theorem III.3 of \cite{Pro04}, and further we have
$\mathbb P^{m,x}(A \cap B) = 0$.
Therefore, $X$ is continuous at $\bar \tau$ almost surely in $\mathbb P^{m,x}$. Together with \eqref{eq:06v5}, we conclude $\mathbb P^{m, x} (\hat \Gamma_{O} ) =1$.

\textbf{Step 2.} Recall that 
$\hat \Gamma_{O}$ and $\Pi_{O}$ are defined by \eqref{eq:Gammah} and \eqref{eq:PiO}, respectively.
We will show that $f_{1}, f_{2}$ are continuous
at all $\omega\in \hat \Gamma_{O}$,  where
$$f_{1}(\omega) = \int_{0}^{T_{O^{c}}(\omega)} e^{- s}\ell(\omega_{s}) ds, 
\ \hbox{ and } \ f_{2}(\omega) = e^{- T_{O^{c}}(\omega)} g(\Pi_{O}(\omega)), \
 \forall \omega\in \mathbb D^{d}_{\infty}.$$
The continuity of $f_{2}$ is the direct consequence of Theorem \ref{t:hit01} and Theorem \ref{t:hit02}. So, it remains to show the continuity of $f_{1}$.

Suppose $\omega^{n} \to \omega \in \hat \Gamma_{O}$ in Skorohod metric, 
and we denote $T_{n} = T_{O^{c}}(\omega^{n})$ 
and $T =  T_{O^{c}}(\omega)$, we conclude  $f_{1}(\omega^{n}) \to f_{1}(\omega)$ as $n\to \infty$, since
\begin{enumerate}
 \item  $T_{n} \to T$ due to  Theorem~\ref{t:hit01};
 \item 
   $\omega^{n} \to \omega$ in $\mathbb D^{d}_{\infty}$ means that
 $\omega^{n}(s) \to \omega(s)$ for all $s\in D_{\omega}^{c}$, where 
 $D_{\omega}^{c}$ is the complement of the countable set 
 $$D_{\omega} := \{s\in (0,\infty): \omega \hbox{ is discontinuous at } s\}.$$
Together with
 the continuity of $\ell$, we have 
 $\ell(\omega^{n}(s)) \to \ell(\omega(s))$ almost everywhere on $(0, t)$
 w.r.t. Lebesgue measure.
\item Finally, we have, as $n\to \infty$
$$|f_{1}(\omega^{n}) - f_{1}(\omega) | 
\le  \int_{0}^{T_{n}} e^{-qs} |\ell(\omega^{n}(s)) - 
\ell(\omega(s))| ds + 2K |T_{n} - T| \to 0.$$
\end{enumerate}

\textbf{Step 3.} In this final step we will show that $V_{m}(x_{n}) \to V_{m}(x)$ if $x_{n} \to x\in \bar O$. We first conclude $\mathbb P^{m,x_{n}}$ is weakly convergent to $\mathbb P^{m,x}$, since 
\begin{enumerate}
 \item [] By Theorem 3.2 of \cite{Kun04}, $X$ satisfies
 $$\mathbb E \Big[ \sup_{0\le s \le t} |X_{s}^{m,x_{n}} - X_{s}^{m,x}|^{2} \Big] 
 \le K_{t} |x_{n} -x |^{2} \to 0, \hbox{ as } n\to \infty.$$
 This means $\{X^{m,x_{n}}_{s}: 0\le s \le t\}$ is convergent to $\{X_{s}^{m,x}: 0 \le s\le t\}$ $\mathbb P$-almost surely with respect to $L^{\infty}$, and hence convergent in distribution with respect to Skorohod metric. Weak convergence on any finite time interval implies the weak convergence on the entire time interval by Theorem 16.7 of \cite{Bil99}.
\end{enumerate}
Moreover, in the above two steps, we established $f_{1} + f_{2}$ is 
continuous $\mathbb P^{m,x}$-almost surely.
Then, we apply the continuous mapping theorem and bounded
convergence theorem to obtain 
$$V_{m}(x_{n}) = \mathbb E^{m,x_{n}} [ (f_{1}+f_{2})(X) ] \to 
\mathbb E^{m,x} [(f_{1}+f_{2})(X)]  = V_{m}(x).$$

\hfill $\square$

\subsection{Proofs of Theorems~\ref{t:hit01} and \ref{t:hit02}}\label{sec:secmnrs}
\subsubsection{Sufficiency of working in simpler topologies}
Let $\Lambda_{\infty}$ be the set of continuous and strictly increasing maps of
$[0, \infty)$ to itself. 
Let 
$$
\|\omega\|_{m} = \sup_{0\le t \le m} |\omega(t)|, \ \|\omega\| = \sup_{0\le t < \infty} |\omega(t)|.$$
The topology induced by the above supnorm is finer than Skorohod topology. Therefore, the continuity of $\Pi_{O}$ at $\omega$
with respect to Skorohod topology automatically implies the continuity
with respect to uniform topology. In this below, we will prove that 
the converse is also true:
the continuity with respect to uniform topology implies the continuity of $\Pi_{O}$ with respect to Skorohod topology. This enables us to simplify
our subsequent analysis by working on a series of simpler metrics.

\begin{lemma}
 \label{p:hit01}
 $T_{O^{c}}(\omega \circ \lambda) = \lambda^{-1} \circ T_{O^{c}}(\omega)$ for all $\omega\in \mathbb D^{d}_{\infty}$ and $\lambda\in \Lambda_{\infty}$.
\end{lemma}
\begin{proof}
 $T_{O^{c}}(\omega \circ \lambda) = \inf\{t>0: \omega\circ \lambda(t) \notin O\}
 = \lambda^{-1} (\inf\{\lambda (t)>0: \omega(\lambda(t)) \notin O\} ) = \lambda^{-1} \circ T_{O^{c}}(\omega).$
\end{proof}
\begin{lemma}\label{l:hit01}
\begin{enumerate}
 \item  If $T_{O^{c}}^{m}$ is  lower semicontinuous  w.r.t $\|\cdot\|_{m}$ for all integer $m$, then
 $T_{O^{c}}$ is  lower semicontinuous  w.r.t $d_{\infty}^{o}$.
 \item If $T_{O^{c}}^{m}$ is upper semicontinuous w.r.t $\|\cdot\|_{m}$ for all integer $m$, then
 $T_{O^{c}}$ is upper semicontinuous w.r.t $d_{\infty}^{o}$.
\end{enumerate}
\end{lemma}
\begin{proof}
 We assume $T_{O^{c}}(\omega) \in (0,\infty)$, otherwise it's obvious.
 Let $\lim_{n} d^{o}_{\infty} (\omega_{n}, \omega) = 0.$ By Theorem 16.1 of \cite{Bil99}, 
 there exists $\lambda_{n}\in \Lambda_{\infty}$ such that
 $$\lim_{n} \|\lambda_{n} - 1 \| = 0$$
 and
 $$\lim_{n} \|\omega_{n}\circ \lambda_{n} - \omega\|_{m} = 0, \ \forall m\in \mathbb N.$$
\begin{enumerate}
 \item  We suppose $T_{O^{c}}^{m}$ is  lower semicontinuous  w.r.t. $\|\cdot\|_{m}$ for every integer $m$. Then, we have
 $$\lim\inf_{n} T_{O^{c}}^{m} (\omega_{n}\circ \lambda_{n}) \ge T_{O^{c}}^{m}(\omega).$$
 Also, we have by Lemma~\ref{p:hit01}
 $$|T_{O^{c}}^{m}(\omega_{n}) - T_{O^{c}}^{m}(\omega_{n}\circ \lambda_{n})| 
 = |T_{O^{c}}(\omega_{n}) \wedge m - \lambda_{n}^{-1}\circ T_{O^{c}}(\omega_{n}) \wedge m| 
 \le \|1 - \lambda_{n}^{-1}\| \to 0.$$
 Thus, we have
 $$\lim\inf_{n} T_{O^{c}}^{m}(\omega_{n}) \ge T_{O^{c}}^{m}(\omega).$$
 Therefore, for a big enough $m$ such that $m > T_{O^{c}}(\omega)$ holds, we have
 $$\lim\inf_{n} T_{O^{c}}(\omega_{n}) \ge \lim\inf_{n}  T_{O^{c}}^{m}(\omega_{n})
 \ge  T_{O^{c}}^{m}(\omega) = T_{O^{c}}(\omega).$$
 This implies $T_{O^{c}}$ is also  lower semicontinuous  w.r.t. $d_{\infty}^{o}$.
 
 \item We suppose $T_{O^{c}}^{m}$ is upper semicontinuous w.r.t. $\|\cdot\|_{m}$ for every integer $m$. Then, we have
 $$\lim\sup_{n} T_{O^{c}}^{m} (\omega_{n}\circ \lambda_{n}) \le T_{O^{c}}^{m}(\omega).$$
 Also, we have
 similarly $|T_{O^{c}}^{m}(\omega_{n}) - T_{O^{c}}^{m}(\omega_{n}\circ \lambda_{n})| \to 0$
 as $n\to \infty$ by Lemma~\ref{p:hit01}, and conclude
 $$\lim\sup_{n} T_{O^{c}}^{m}(\omega_{n}) \le T_{O^{c}}^{m}(\omega) \le T_{O^{c}}(\omega) \quad
 \hbox{ for all integer } m.$$
 Now we fix an integer $m > T_{O^{c}}(\omega) +1$. 
 This means, $\forall \epsilon \in (0,1)$, there exists $N_{\epsilon}$ such
 that
 $$T_{O^{c}}(\omega_{n}) \wedge m \le T_{O^{c}}(\omega) + \epsilon, \quad \forall n \ge N_{\epsilon}.$$
 Since $m > T_{O^{c}}(\omega) + \epsilon$, the left hand side $T_{O^{c}}(\omega_{n}) \wedge m$ must be equal to $T_{O^{c}}(\omega_{n})$, i.e. 
  $$T_{O^{c}}(\omega_{n}) \le T_{O^{c}}(\omega) + \epsilon, \quad \forall n \ge N_{\epsilon}.$$
 This implies $T_{O^{c}}$ is also upper semicontinuous w.r.t. $d_{\infty}^{o}$.
\end{enumerate}
\end{proof}

\subsubsection{The problem in dimension one}
In this below, we will identify the continuity set in one dimensional c\`adl\`ag space for the mapping $T_{O^{c}}$ with respect to uniform topology
induced by supnorm.
\begin{lemma}
 \label{l:hit03}
 The mapping $\omega \mapsto T_{(-\infty, 0)}^{m}(\omega) $ is upper semicontinuous in $\mathbb D^{1}_{\infty}$ w.r.t. $\|\cdot\|_{m}$  for every $m\in \mathbb N$.
\end{lemma}
\begin{proof}
For convenience, we denote $\hat T_{m}(\omega) = T_{(-\infty, 0)} (\omega) \wedge m$.
 It's enough to show that
\begin{center}
 If $\|\omega_{n}-\omega\|_{m} \to 0$, then $\lim\sup_{n} \hat T_{m}(\omega_{n}) \le \hat T_{m}(\omega)$. 
\end{center}
We prove it in two cases separately:
\begin{enumerate}
 \item Assume $\inf_{0\le t \le m} \omega(t) >0$. This implies $\hat T_{m}(\omega) = m$. 
 Given $\|\omega_{n}-\omega\|_{m} \to 0$, there exists $N$, such that
 $$\forall n> N, \ \|\omega_{n}-\omega\|_{m} < \frac 1 2 \inf_{0\le t \le m} \omega(t).$$
 This yields
 $$\forall n> N, \ \forall s\in [0,m], \ 
 \omega_{n}(s)-\omega(s) > -\frac 1 2 \inf_{0\le t \le m} \omega(t).$$
 Therefore, 
 $$\forall n> N, \ \forall s\in [0,m], \ 
 \omega_{n}(s)>0,$$
 or equivalently, $\hat T_{m}(\omega_{n}) = m$ for all $n>N$.
 This proves the conclusion of the first case.

 \item Assume $\inf_{0\le t \le m} \omega(t) \le 0$. Fix arbitrary $\epsilon>0$, then
 $$\exists t_{\epsilon} \in [\hat T_{m}(\omega), \hat T_{m}(\omega) + \epsilon) 
 \hbox{ such that } \omega(t_{\epsilon})< 0.$$
 Given  $\|\omega_{n}-\omega\|_{m} \to 0$, 
 $$\exists N \ \hbox{ such that } \ \|\omega_{n}-\omega\|_{m} < \frac 1 2 |\omega(t_{\epsilon})|, \ 
 \forall n\ge N.$$
 In particular, one can write
 $\omega_{n}(t_{\epsilon}) - \omega(t_{\epsilon}) < - \frac 1 2 \omega(t_{\epsilon})$, 
 or equivalently
 $$\exists N \hbox{ such that } \ \omega_{n}(t_{\epsilon}) < 0,  \ 
 \forall n\ge N.$$
Therefore, $\hat T_{m}(\omega_{n}) \le t_{\epsilon} \le \hat T_{m}(\omega) + \epsilon$ for
all $n\ge N$. By taking $\lim\sup_{n}$ both sides, we have
$$\lim\sup_{n} \hat T_{m}(\omega_{n}) \le \hat T_{m}(\omega) + \epsilon$$
and the conclusion follows due to the arbitrary selection of $\epsilon$.
\end{enumerate}

\end{proof}

\begin{lemma}
 \label{l:hit04}
 $\omega\mapsto T_{(-\infty, 0]}^{m} (\omega_{*})$ is  lower semicontinuous  in $\mathbb D^{1}_{\infty}$ w.r.t. $\|\cdot\|_{m}$ for every
 $m\in \mathbb N$, where $$\omega_{*} (t) = \lim\inf_{s\to t} \omega(s), \ \forall t>0$$ is 
 the lower envelope of $\omega$.
\end{lemma}
\begin{proof}
For simplicity, we denote 
$$\tilde T_{m}(\omega) = T_{(-\infty, 0]} (\omega_{*}) \wedge m 
\hbox{ and }
M[\omega](t) = \inf_{0\le s \le t} \omega(s).$$ Note that $M[\omega] = M[\omega_{*}]$ is a 
non-increasing process. It's enough to show that
\begin{center}
 If $\|\omega_{n}-\omega\|_{m} \to 0$, then $\lim\inf_{n} \tilde T_{m}(\omega_{n}) 
 \ge \tilde T_{m}(\omega)$. 
\end{center}
\begin{enumerate}
 \item Assume $\tilde T_{m}(\omega) = m$. This implies 
 $M[\omega](m) = M[\omega_{*}](m) >0$, otherwise $\tilde T_{m}(\omega) < m$.  
 Given $\|\omega_{n}-\omega\|_{m} \to 0$,
 $$\exists N, \hbox{ such that } \|\omega_{n}-\omega\|_{m}  < \frac 1 2 M[\omega](m), \
 \forall n \ge N,$$
 which implies, there exists $N$ such that
 $$\omega_{n}(t) > \omega(t) -  \frac 1 2 M[\omega](m)
 \ge \frac 1 2 M[\omega](m)>0, \ \forall t \in (0, m), \ \forall n \ge N.$$
 Hence, $\tilde T_{m}(\omega_{n}) = m$ for all $n\ge N$, and this proves
 the continuity
 at $\omega$ for this case.
 \item Assume $\tilde T_{m}(\omega) < m$. Since  $\omega_{*}$ is  lower semicontinuous , we have
 $$M[\omega](\tilde T_{m}(\omega)) \le 0, \hbox{ and } M[\omega](t) >0, \forall t < \tilde T_{m}(\omega).$$
Fix arbitrary $\epsilon>0$. Then, we have 
$M[\omega] (\tilde T_{m}(\omega) - \epsilon) >0$, and 
$$\exists N, \hbox{ such that } \|\omega_{n} - \omega \|_{m} < \frac 1 2 M[\omega](\tilde T_{m}(\omega) - \epsilon), \ \forall n\ge N.$$
This leads to, for all $n\ge N$ and $t < \tilde T_{m}(\omega) - \epsilon$
$$\omega_{n}(t) > \omega (t) - \frac 1 2 M[\omega](\tilde T_{m}(\omega) - \epsilon)
\ge \frac 1 2 M[\omega](\tilde T_{m}(\omega) - \epsilon)>0.
$$
In other words, we have 
$\tilde T_{m}(\omega_{n}) \ge \tilde T_{m}(\omega) - \epsilon$ for all $n \ge N$. 
So we conclude $\lim\inf_{n} \tilde T_{m}(\omega_{n}) 
 \ge \tilde T_{m}(\omega)$ for the this case.
\end{enumerate}
\end{proof}
\begin{lemma}
 \label{p:hit03}
\begin{enumerate}
\item   $T_{(-\infty, 0]}^{m}$ is upper semicontinuous on 
$$\{\omega \in \mathbb D^{1}_{\infty}: 
T_{(-\infty, 0]}^{m}(\omega) = T_{(-\infty, 0)}^{m}(\omega) \} \hbox{ w.r.t. } \|\cdot\|_{m};$$
\item  $T_{(-\infty, 0]}^{m}$ is  lower semicontinuous  on 
$$\{\omega \in \mathbb D^{1}_{\infty}: 
T_{(-\infty, 0]}^{m}(\omega) = T_{(-\infty, 0]}^{m}(\omega_{*}) \} \hbox{ 
w.r.t. } \|\cdot\|_{m}.$$
\end{enumerate}
\end{lemma}
\begin{proof}
If (a) $\omega_{n} \to \omega$ w.r.t. $\|\cdot\|_{m}$; and 
(b) $T_{(-\infty, 0]}^{m}(\omega) = T_{(-\infty, 0)}^{m}(\omega) $, then
Lemma~\ref{l:hit03} implies 
$$\lim_{n}\sup T^{m}_{(-\infty, 0]}(\omega_{n}) 
\le \lim_{n}\sup T^{m}_{(-\infty, 0)}(\omega_{n})
\le T^{m}_{(-\infty, 0)}(\omega) =  T^{m}_{(-\infty, 0]}(\omega),$$ 
which asserts the upper semicontinuity.

Similarly, if (a) $\omega_{n} \to \omega$ w.r.t. $\|\cdot\|_{m}$; and 
(b) $T_{(-\infty, 0]}^{m}(\omega) = T_{(-\infty, 0]}^{m}(\omega_{*}) $, then
Lemma~\ref{l:hit04} implies 
$$\lim_{n}\inf T^{m}_{(-\infty, 0]}(\omega_{n}) 
\ge \lim_{n}\inf T^{m}_{(-\infty, 0]}(\omega_{n,*})
\ge T^{m}_{(-\infty, 0]}(\omega_{*}) =  T^{m}_{(-\infty, 0]}(\omega),$$ 
which asserts  the lower semicontinuity.
\end{proof}

\subsubsection{Proof of Theorem \ref{t:hit01}} \label{sec:pfhit01}
\hfill

\textbf{Step 1.} The proof relies on a dimension reduction.
Let us define
the signed distance function 
\begin{equation} \label{eq:rho2}
 \rho(x) = \left \{
\begin{array}
 {ll}
 dist (x, \partial O) & \hbox{ if } x\in O; \\
 - dist (x, \partial O) & \hbox{ otherwise}
\end{array} \right.
 \end{equation}
 
Note that,  if $O$ is open, then
 $$T_{O^{c}}(\omega) = \inf\{t\ge 0: \omega(t) \notin O\} = \inf\{t\ge 0: \rho \circ \omega(t) \le 0\} = 
 T_{(-\infty, 0]}(\rho\circ \omega),$$
 and 
 $$T_{\bar O^c}(\omega) = \inf\{t\ge 0: \omega(t) \notin \bar O\} = \inf\{t\ge 0: \rho \circ \omega(t) < 0\} = 
 T_{(-\infty, 0)}(\rho\circ \omega).$$
In other words, we have 
\begin{equation}\label{eq:dd}
 T_{O^{c}} = T_{(-\infty, 0]} \circ \rho, \  T_{\bar O^c} = T_{(-\infty, 0)} \circ \rho,\ 
 \forall \omega \in \mathbb D_{\infty}^{d}
 \hbox{ for all open set } O. 
\end{equation}
This simple fact enables us to generalize 1-d result of Lemma~\ref{p:hit03} to the multidimensional case.
 
\textbf{Step 2.}
 First assume $d = 1$ and $O = (0, \infty)$.  Lemma~\ref{p:hit03} and
 Lemma~\ref{l:hit01} implies $T_{(-\infty, 0]}$ is continuous on 
 $$B = \{\omega\in \mathbb D^{1}_{\infty}: T_{(-\infty, 0]}(\omega_{*}) =
  T_{(-\infty, 0]}(\omega) =  T_{(-\infty, 0)}(\omega)\}.$$
 Recall that we want to show $T_{(-\infty, 0]}$ is continuous on  
 $$\Gamma_{(0,\infty)} = \{\omega\in \mathbb D^{1}_{\infty}: T_{(-\infty, 0]}(\omega^{-}) =
  T_{(-\infty, 0]}(\omega) =  T_{(-\infty, 0)}(\omega)\}.$$
  Hence, it's enough to show $B = \Gamma_{(0,\infty)}$.  
\begin{enumerate}
 \item By an inequality of $ T_{(-\infty, 0]}(\omega_{^{-}*}) \le T_{(-\infty, 0]} (\omega^{-})
 \le T_{(-\infty, 0]} (\omega)$, we have  $B \subset \Gamma_{(0,\infty)}$.
 \item If there exists $\omega\in \Gamma_{(0,\infty)} \setminus B$, then 
 $T_{(-\infty, 0]}(\omega_{*}) < T_{(-\infty, 0]}(\omega^{-})$. 
 This yields that
 $$\omega_{*}(T_{(-\infty, 0]}(\omega_{*})) \le 0 < \omega^{-} (T_{(-\infty, 0]}(\omega_{*})) ,$$
 which again implies, with the notion of $\Delta \omega(t) = \omega(t) - \omega(t-)$
 $$\Delta \omega(T_{(-\infty, 0]}(\omega_{*})) < 0, \quad  \omega(T_{(-\infty, 0]}(\omega_{*})) = 
 \omega_{*}(T_{(-\infty, 0]}(\omega_{*})) \le 0.$$
 Hence, we have $T_{(-\infty, 0]}(\omega) = T_{(-\infty, 0]}(\omega_{*})$, which is
  a contradiction to $\omega\notin B$.
\end{enumerate}
In conclusion, we obtain $B = \Gamma_{(0,\infty)}$ and
$T_{(-\infty, 0]}$ is continuous at any $\omega\in \Gamma_{(0,\infty)}$. 

\textbf{Step 3.}
Now we turn to the general case of $d\ge 1$. If 
$ \omega_{n} \to \omega \in \Gamma_{O}$, then
$\rho\circ \omega_{n} \to \rho \circ \omega \in \Gamma_{(0,\infty)}$ by the continuity of
$\rho$. Thanks to \eqref{eq:dd} and the continuity of $T_{(-\infty, 0]}$ on 
$\Gamma_{(0,\infty)}$, we conclude,
$$T_{O^{c}}(\omega_{n}) = T_{(-\infty, 0]}(\rho(\omega_{n})) \to
T_{(-\infty, 0]}(\rho(\omega)) = T_{O^{c}}(\omega).$$
\hfill $\square$

\subsubsection{Proof of Theorem \ref{t:hit02}}\label{sec:pfhit02}

 Let $\omega^{n} \to \omega \in \hat \Gamma_{O}$ 
 in Skorohod topology, and denote for simplicity that
 $$T = T_{O^{c}}(\omega), T_{n} = T_{O^{c}} (\omega^{n}).$$
 Then, we can write $\omega(T)  = \Pi_{O}(\omega)$ and $\omega(T_{n}) = 
 \Pi_{O}(\omega^{n})$. We want to show that
 $\omega(T_{n}) \to \omega(T)$ as $n\to \infty$.
\begin{enumerate}
 \item If $\Pi_{O} (\omega^{-}) = \Pi_{O}(\omega)$, then $\Pi_{O} (\omega) \in \partial O$. Since $\omega$ is continuous at $T$, $\omega^{n} \to \omega$ 
 in Skorohod metric implies that $\omega^{n} \to \omega$ uniformly on some interval $(T-\epsilon, T+\epsilon)$ for $\epsilon>0$, i.e.
 $$\sup_{|s-T| < \epsilon} |\omega^{n}(s) - \omega(s)| \to 0, 
 \hbox{ as } n\to \infty.$$
 Sine $T_{n} \to T$ by Theorem~\ref{t:hit01}, there exists $N$ such that
 $T_{n} \in (T-\epsilon, T+ \epsilon)$ for all $n\ge N$. Together with
 the continuity of $\omega$ at $T$, we conclude that
 $$ 
\begin{array}{ll}
 |\omega^{n} (T_{n}) - \omega (T)| & \le 
 |\omega^{n}(T_{n}) - \omega(T_{n})| + |\omega(T_{n}) - \omega(T)|
 \\ & \le \sup_{|s-T| < \epsilon} |\omega^{n}(s) - \omega(s)|  + 
 |\omega(T_{n}) - \omega(T)| 
 \\ & 
 \to 0, \hbox{ as } n\to 0
\end{array}
 $$
 
\item 
If $\Pi_{O} (\omega^{-}) \neq \Pi_{O}(\omega)$, then $\omega \in \hat \Gamma_{O}$
means that $\omega^{-}(T) \in O$ and $\omega(T) \in O^{c}$. 
\begin{enumerate}
 \item \label{list:2a} If $\|\omega^{n} - \omega\|_{m} \to 0$ for some $m>T+1$, then there 
 exists $N_{1}$ such that $T_{n} < m$ for all $n\ge N_{1}$.
 Since $T_{O^{c}} (\omega^{-}) = T_{O^{c}}(\omega)$, we can also define
 $$\epsilon := \sup_{0\le s \le T} \rho(\omega^{-}(s)) >0,$$
 where $\rho$ is the signed distance to the boundary as of \eqref{eq:rho2}.
 Note that, there exists $N_{2} > N_{1}$ such that
 $$\|\omega^{n} - \omega \|_{m} < \frac 1 2 \epsilon, \ \forall n > N_{2}.$$
 Therefore, $\sup_{0\le s \le T} \rho(\omega^{n}(s)) >0$ and 
 $T_{n} \ge T$. Hence, $T_{n} \downarrow T$ as $n\to \infty$, and the
 right continuity of $\omega$ leads to 
 $$ 
\begin{array}{ll}
 |\omega^{n} (T_{n}) - \omega (T)| & \le 
 |\omega^{n}(T_{n}) - \omega(T_{n})| + |\omega(T_{n}) - \omega(T)|
 \\ & \le \sup_{|s-T| < \epsilon} |\omega^{n}(s) - \omega(s)|  + 
 |\omega(T_{n}) - \omega(T)| 
\\ & \to 0, \hbox{ as } n\to 0
\end{array}
 $$
 \item 
 If $d^{o}_{\infty} (\omega^{n}, \omega) \to 0$, then there exists $\lambda_{n}\in \Lambda_{\infty}$ such that
 $$\lim_{n} \|\lambda_{n} - 1 \| = 0$$
 and
 $$\lim_{n} \|\omega^{n}\circ \lambda_{n} - \omega\|_{m} = 0, \ \forall m\in \mathbb N.$$
 Applying Lemma~\ref{p:hit01}, we have
 $$\omega^{n} (T_{O^{c}}(\omega^{n})) =
  \omega^{n} (\lambda_{n} \circ T_{O^{c}} (\omega^{n} \lambda_{n}))
  = \hat \omega^{n} (T_{O^{c}}(\hat \omega^{n})), $$
  where $\hat \omega^{n} = \omega^{n}\circ \lambda_{n}$.
  Since $\lim_{n}\|\hat \omega^{n} - \omega \|_{m} = 0$ for all $m\in \mathbb N$, 
  we can repeat the same proof of Step~\ref{list:2a}, and obtain
  $\hat \omega^{n} (T_{O^{c}}(\hat \omega^{n})) \to  \omega(T)$, which in turn
  implies that $\omega^{n}(T_{n}) \to \omega(T)$.

\end{enumerate}
 
\end{enumerate}

\hfill $\square$

\appendix

\section{Equivalence of Definition \ref{d:01v5} and Definition 2 of \cite{BI08}} 
\label{s:def}

\subsection{Closure of the test function space}

Recall that test function spaces $J^{\pm}(u,x)$ were defined in \eqref{eq:J+} and \eqref{eq:J-}.
Next, we shall define 
the closure of test function space $J^{\pm}(u,x)$
in the sense of
non-local version of closure of semijets of \cite{CIL92}, and provide the sufficient condition for a function $\phi$ to be
 in the closure $\bar J^{\pm}(u,x)$.

\begin{definition}
 \label{d:02v5}
 A set $\bar J^{+} (u, x)$ {\rm (respectively $\bar J^{-} (u, x)$)} is given by all functions $\phi \in C_{x}$ satisfying the following conditions: There exists $x_{\epsilon} \to x$ and $\phi_{\epsilon} \in 
 J^{+} (u, x_{\epsilon})$ {\rm (respectively 
 $\phi_{\epsilon} \in J^{-} (u, x_{\epsilon})$)} satisfying
 $$(x_{\epsilon}, \phi_{\epsilon}(x_{\epsilon}), D \phi_{\epsilon}(x_{\epsilon}), D^{2} \phi_{\epsilon} (x_{\epsilon}), \mathcal I (\phi_{\epsilon}, x_{\epsilon})) \to 
 (x, \phi(x), D \phi(x), D^{2} \phi (x), \mathcal I (\phi, x)).$$
\end{definition}

For notational simplicity, we define a shifted L\'{e}vy measure 
$\nu_{x}$ by 
$\nu_{x}(dy) = \hat \nu(y-x) dy$ for any $x\in \mathbb R^{d}$.
Accordingly, we say $\phi \in L^{1}(\nu_{x}, B)$ for some Lebesgue 
measurable set
$B$ of $\mathbb R^{d}$, if $\int_{B} |\phi(y)| \nu_{x} (dy) <\infty$ is well defined.

\begin{lemma}
 \label{l:02v5}
 For a given $x\in \mathbb R^{d}$ and $\phi \in C_{x}$, 
 if there exists $\{(\phi_{\epsilon}, x_{\epsilon}): \epsilon >0\}$ and $r>0$ such that
\begin{enumerate}
 \item $\lim_{\epsilon} x_{\epsilon} = x$;
 \item $\phi_{\epsilon} \in C^{\infty}(B_{2r}(x))$ such that
 $\|\phi_{\epsilon} - \phi\|_{W^{2, \infty}(B_{r}(x))} \to 0$ as $\epsilon \to 0$;
 
 \item  $\exists \hat \phi  \in L^{1}(\nu_{x}, B_{r}^{c}(x))$ such that 
 $|\phi_{\epsilon}| \le \hat \phi$  and $\lim_{\epsilon \to 0} \|\phi_{\epsilon} - \phi\|_{L^{1}(\nu_{x}, B_{r}^{c}(x))} = 0$;
\end{enumerate}
Then, we have, 
$$\mathcal I_{ r, 1} (\phi_{\epsilon}, x_{\epsilon}) \to 
 \mathcal I_{ r, 1} (\phi, x), \hbox{ and }
 \mathcal I_{ r, 2} (\phi_{\epsilon}, x_{\epsilon}) \to 
 \mathcal I_{r, 2} (\phi, x), \ \hbox{ as }\epsilon \to 0^{+}.$$
\end{lemma}
\begin{proof}
Without loss of generality, we assume $r$ is small enough such that $\phi \in C^{\infty}(B_{2r}(x))$. 
For an arbitrary $\epsilon$ satisfying $|x_{\epsilon} - x| < r/3$, using $f_{\epsilon}$ defined by
 $$f_{\epsilon} (y) = \phi_{\epsilon}(x_{\epsilon} + y) - \phi (x+ y),$$
 we can write the following inequalities:
 $$
\begin{array}{ll}
 \Big | \mathcal I_{r, 1} (\phi_{\epsilon}, x_{\epsilon}) - 
 \mathcal I_{r, 1} (\phi, x) \Big |
& 
 = 
\Big | \int_{B_{r}} (f_{\epsilon} (y) - f_{\epsilon} (0) - D f_{\epsilon}(0) \cdot y) \nu(dy) \Big |
\\ &
 \le \frac 1 2 \|D^{2} f_{\epsilon}\|_{L^{\infty}(\bar B_{r})} 
 \int_{B_{r}} |y|^{2} \nu (dy).
\end{array}
 $$
 Note that $x_{\epsilon} + y \in \bar B_{r}(x)$ whenever $y\in B_{ r}$. 
\begin{itemize}
\item   Since $D^{2}\phi_{\epsilon} \to D^{2} \phi$ 
holds uniformly in $B_{r}(x)$, we have 
$$ \sup_{y \in B_{r}} | D^{2} \phi_{\epsilon}(x_{\epsilon} + y) - D^{2} \phi(x_{\epsilon} + y)| \to 0^{+};$$ 
\item  $\phi \in C^{\infty}(B_{2r})$ implies that $D^{2} \phi$ is uniformly 
continuous in $B_{r}$ and 
$$ \sup_{y \in B_{r}} | D^{2} \phi(x_{\epsilon} + y) - D^{2} \phi(x_{\epsilon} + y)| \to 0^{+};$$ 
\end{itemize}
 we conclude that
 $ \frac 1 2 \|D^{2} f_{\epsilon}\|_{L^{\infty}(\bar B_{r})} \to 0$ and
 $\mathcal I_{r, 1} (\phi_{\epsilon}, x_{\epsilon}) \to 
 \mathcal I_{r, 1} (\phi, x)$ as $\epsilon \to 0^{+}$.
 
 Next, we write 
 $$|\mathcal I_{ r, 2} (\phi_{\epsilon}, x_{\epsilon}) - 
 \mathcal I_{ r, 2} (\phi, x)|  \le  TERM1 + TERM2 + TERM3,$$
 where three terms are followed by
\begin{enumerate}
 \item Due to the property of L\'{e}vy measure, it yields $\nu(B_{r}^{c}) <\infty$, and uniform convergence of $\phi_{\epsilon}$ on $B_{2r}(x)$ leads to
 $$
\begin{array}
 {ll}
  TERM1 &= \Big | \int_{B_{r}^{c}}(\phi_{\epsilon}(x_{\epsilon}) - \phi(x) ) \nu(dy) \Big | 
 \\&
 = | \phi_{\epsilon}(x_{\epsilon}) - \phi(x) | 
 \nu(B_{r}^{c}) \to 0,  \hbox{ as } \epsilon \to 0^{+};
\end{array}
$$
 \item  Since $\hat \nu \in C_{b}(B_{r}^{c})$, we have
 $$TERM2 = \Big |\int_{B_{ r}^{c}}(\phi_{\epsilon} - \phi) (x + y ) \nu(dy) \Big | \le 
 \|\phi_{\epsilon} - \phi\|_{L^{1}(\nu_{x}, B_{r}^{c}(x))} \to 0, \hbox{ as } \epsilon \to 0^{+};$$
\item One can write 
$$ 
\begin{array}
 {ll}
 TERM3  
 & \displaystyle
 = \Big |\int_{B_{ r}^{c}}(\phi_{\epsilon} (x_{\epsilon} + y) - 
 \phi_{\epsilon} (x + y )) \nu(dy) \Big |
 \\   & \displaystyle
 = \Big |\int_{B_{ r}^{c}(x_{\epsilon})} \phi_{\epsilon} (z) 
 \hat \nu (z - x_{\epsilon}) dz - 
 \int_{B_{ r}^{c}(x)} \phi_{\epsilon} (z) \hat \nu(z -x) d z \Big | 
 \\ & \displaystyle
\le TERM 31 + TERM32 + TERM33
\end{array}$$
where $TERM3$ is again divided by three terms as such:
\begin{itemize}
 \item Since $|\phi_{\epsilon}| \le \hat \phi \in L^{1}(\nu_{x}, B_{r}^{c}(x))$, 
 $\hat \nu \in C_{b} (B_{r}^{c})$ and $|z - x_{\epsilon}| \wedge |z-x| \ge r$, one can use Dominated Convergence Theorem to conclude that
 $$TERM31 = \int_{B_{ r}^{c}(x_{\epsilon}) \cap B_{ r}^{c}(x)} 
 |\phi_{\epsilon} (z) (\hat \nu(z - x_{\epsilon}) - \hat \nu(z -x))| dz \to 0$$
 as $\epsilon \to 0$;
 \item Note that $x_{\epsilon} + y \in B_{r}(x)$ whenever $y \in B_{ r}^{c} \cap B_{ r}(x-x_{\epsilon})$. Together with
 $\|\phi_{\epsilon}\|_{L^{\infty}(B_{r}(x))} \to \|\phi\|_{L^{\infty}(B_{r}(x))}$
 as $\epsilon \to 0$ due to the uniform convergence on $B_{2r}(x)$,  it
 yields
 $$ 
\begin{array}
 {ll}
 TERM32  
 & \displaystyle
 = \int_{B_{ r}^{c}(x_{\epsilon}) \cap B_{ r}(x)} 
 |\phi_{\epsilon} (z)|  \hat \nu (z - x_{\epsilon}) dz 
 \\ & \displaystyle
= \int_{B_{ r}^{c} \cap B_{ r}(x-x_{\epsilon})} 
 |\phi_{\epsilon} (x_{\epsilon} + y)|  \hat \nu (y) d y
  \\ & \displaystyle
  \le \| \phi_{\epsilon}\|_{L^{\infty}(B_{r}(x))} \nu(B_{ r}^{c} \cap B_{ r}(x-x_{\epsilon})) \to 0,  \hbox{ as } \epsilon \to 0^{+};
\end{array}$$
\item Similarly, we have 
$x + y \in B_{r}(x_{\epsilon}) \subset B_{4r/3}(x)$ whenever $y \in B_{ r}^{c} \cap B_{ r}(x_{\epsilon} -x)$. Thus, we have
$$\| \phi_{\epsilon}\|_{L^{\infty}(B_{r}(x_{\epsilon}))}  \le 
\| \phi_{\epsilon}\|_{L^{\infty}(B_{2r}(x))}  \to 
\| \phi \|_{L^{\infty}(B_{2r}(x))}   \ \hbox{ as } \epsilon \to 0$$
due to the uniform convergence on $B_{2r}(x)$,
and it yields
 $$ 
\begin{array}
 {ll}
 TERM33  
 & \displaystyle
 = \int_{B_{ r}(x_{\epsilon}) \cap B^{c}_{ r}(x)} 
 |\phi_{\epsilon} (z)|  \hat \nu (z - x) dz 
 \\ & \displaystyle
= \int_{B_{ r} (x_{\epsilon} -x) \cap B_{ r}^{c}} 
 |\phi_{\epsilon} (x + y)|  \hat \nu (y) d y
  \\ & \displaystyle
  \le \| \phi_{\epsilon}\|_{L^{\infty}(B_{r}(x_{\epsilon}))} 
  \nu(B_{r} (x_{\epsilon} -x) \cap B_{ r}^{c})
\\ & \displaystyle
  \le \| \phi_{\epsilon}\|_{L^{\infty}(B_{2r}(x))} 
  \nu(B_{r} (x_{\epsilon} -x) \cap B_{ r}^{c}) \to 0  \hbox{ as } \epsilon \to 0^{+};
\end{array}$$
\end{itemize}
Therefore, $TERM3$ is also converging to zero as $\epsilon$ goes to zero.
\end{enumerate} 
This completes the proof of 
$|\mathcal I_{ r, 2} (\phi_{\epsilon}, x_{\epsilon}) - 
 \mathcal I_{ r, 2} (\phi, x)| \to 0$.
 
\end{proof}

Now we can simplify the statement of Lemma \ref{l:02v5} for the convenience of the later use.
\begin{lemma}
 \label{p:02v6}
  For a given $x\in \mathbb R^{d}$ and $\phi \in C_{x}$, 
 if there exists $\{(\phi_{\epsilon}, x_{\epsilon}): \epsilon >0\}$ and $r>0$ such that
\begin{enumerate}
 \item $\lim_{\epsilon} x_{\epsilon} = x$;
 \item $\phi_{\epsilon} \in C^{\infty}(B_{r}(x))$ such that
 $\|\phi_{\epsilon} - \phi\|_{W^{2, \infty}}(B_{r}(x)) \to 0$ as $\epsilon \to 0$;
 
 \item  $\exists \hat \phi  \in L^{1}(\nu_{x}, B_{r}^{c}(x))$ such that 
 $|\phi_{\epsilon}| \le \hat \phi$  and $\lim_{\epsilon \to 0} \|\phi_{\epsilon} - \phi\|_{L^{1}(\nu_{x}, B_{r}^{c}(x))} = 0$;
\end{enumerate}
Then, we have, for any $\hat r>0$
\begin{equation}
 \label{eq:11}
 \mathcal I_{\hat  r, 1} (\phi_{\epsilon}, x_{\epsilon}) \to 
 \mathcal I_{\hat  r, 1} (\phi, x), \hbox{ and }
 \mathcal I_{\hat  r, 2} (\phi_{\epsilon}, x_{\epsilon}) \to 
 \mathcal I_{\hat r, 2} (\phi, x), \ \hbox{ as }\epsilon \to 0^{+}.
\end{equation}
\end{lemma}
\begin{proof}
 Let $\hat r = r/2$, then $(\phi_{\epsilon}, x_{\epsilon})$ satisfies all conditions of Lemma \ref{l:02v5} by switching $r$ by $\hat r$ and 
 $\hat \phi$ by 
 $\hat \phi I_{B_{r}^{c}(x)} + (|\phi| + 1) I_{\overline B_{r}(x)}$. Therefore, the conclusion \eqref{eq:11} holds for $\hat r = r/2$. Together with \eqref{eq:03v5}, we have 
 $$\mathcal I(\phi_{\epsilon}, x_{\epsilon}) :=
 \mathcal I(\phi_{\epsilon}, x_{\epsilon}; \nu) \to  \mathcal I(\phi, x) := 
  \mathcal I(\phi, x; \nu).$$
This convergence is valid for all $\nu$, and we apply this convergence to 
$I_{B_{\hat r}}(y) \nu(dy)$, which yields 
 $$\forall \hat r>0, \ \mathcal I_{\hat  r, 2} (\phi_{\epsilon}, x_{\epsilon}) \to 
 \mathcal I_{\hat r, 2} (\phi, x), \ \hbox{ as }\epsilon \to 0^{+}.$$
 This in turn implies, due to \eqref{eq:03v5}
 $$\forall \hat r>0, \ \mathcal I_{\hat  r, 1} (\phi_{\epsilon}, x_{\epsilon}) \to 
 \mathcal I_{\hat r, 1} (\phi, x), \ \hbox{ as }\epsilon \to 0^{+}.$$
\end{proof}

We will give a sufficient condition for $\phi \in \bar J^{\pm} u(x)$ in this below. 
\begin{proposition}
 \label{p:01v5}
\begin{enumerate}
 \item For a given $x\in \mathbb R^{d}, \phi\in C_{x}$ and $u\in USC(\mathbb R^{d})$, 
 if there exists 
 $$\{(\phi_{\epsilon}, x_{\epsilon}): \phi_{\epsilon} \in J^{+} (u, x_{\epsilon}),  \ \epsilon >0\}$$ 
 satisfying all conditions in Lemma \ref{p:02v6},
then we have $\phi\in \bar J^{+} (u, x)$. 

\item For a given $x\in \mathbb R^{d}, \phi\in C_{x}$ and $u\in LSC(\mathbb R^{d})$, 
 if there exists 
 $$\{(\phi_{\epsilon}, x_{\epsilon}): \phi_{\epsilon} \in J^{-} (u, x_{\epsilon}),  \ \epsilon >0\}$$ 
 satisfying all conditions in Proposition \ref{p:02v6},
then we have $\phi\in \bar J^{-} (u, x)$. 
\end{enumerate}

\end{proposition}
\begin{proof}
 $L^{1}$-convergence implies, with a subsequence,  
 $\phi_{\epsilon}\to \phi$ pointwisely, and so $\phi \ge u$. Uniform convergence in $B_{r}(x)$ also implies that 
 $$(x_{\epsilon}, \phi_{\epsilon}(x_{\epsilon}), D \phi_{\epsilon}(x_{\epsilon}), D^{2} \phi_{\epsilon} (x_{\epsilon})) \to 
 (x, \phi(x), D \phi(x), D^{2} \phi (x)).$$
 
 Moreover, $\phi(x) = u(x)$ holds 
 by the facts of $\phi_{\epsilon} \in J^{+} (u, x_{\epsilon})$ and
 upper semicontinuity of $u$, i.e.
 $$\phi(x) = \lim_{\epsilon} \phi_{\epsilon} (x_{\epsilon}) = 
 \lim\sup_{\epsilon} \phi_{\epsilon} (x_{\epsilon}) = 
 \lim\sup_{\epsilon} u (x_{\epsilon}) = u(x).$$
 
 In view of the relation of \eqref{eq:03v5} and Proposition \ref{p:02v6}, 
 we also have 
 $\mathcal I(\phi_{\epsilon}, x_{\epsilon}) \to 
 \mathcal I (\phi, x)$ and $\phi\in \bar J^{+} (u, x)$. 
  Similarly, we can show 
$\phi\in \bar J^{-} (u, x)$.
\end{proof}

Finally, we present the continuity of $\mathcal I (\phi, \cdot)$, which will be
later used several times.
\begin{lemma}
 \label{p:02v5}
 For a given $x\in \mathbb R^{d}$ and $\phi \in C_{x}$, the mapping
 $\mathcal I(\phi,\cdot)$ is continuous at $x$.
 \end{lemma}

\begin{proof}

 If $x_{\epsilon} \to x$, then we can take $\phi_{\epsilon} = \phi$ and 
 apply
 Proposition \ref{p:02v6} and the relation of \eqref{eq:03v5}  to conclude the result.
\end{proof}

\subsection{Proof of equivalence between two definitions}
This section is devoted to the proof of Proposition~\ref{p:01v6}.
\begin{proof}
 If $u$ is a subsolution of Definition \ref{d:11}, then it automatically satisfies
 subsolution properties of  Definition \ref{d:01v5}. 
 In the reverse direction, in view of Assumption
 \ref{a:01v5} (\ref{a:011v5}), 
 we shall show that, arbitrary $\phi \in J^{+} (u, x)$ and $r>0$ implies that
 $$w := \phi I_{\bar B_{r}(x)} + 
 u^{g} I_{\bar B_{r}^{c}(x)} \in \bar J^{+} (u, x),$$
 where we recall that $u^{g}$ is defined in Definition~\ref{d:01v5}.
 In the rest of the proof, 
 we fix $x\in O$ and $r = \frac 1 2 dist (x, \partial O)$.
 According to Proposition \ref{p:01v5}, we shall construct 
 $\{\phi_{\epsilon} \in J^{+}(u,x_{\epsilon}): \epsilon>0\}$ 
 satisfying all conditions of 
 Proposition \ref{p:02v6}. 
 We establish this in the following steps with restriction on 
 $\epsilon \in (0, 1 \wedge \frac {r^{4}}{4})$.
\begin{enumerate}
 \item Set $\hat \phi (y) = \phi(y)  + \sqrt \epsilon |y - x|^{2}$.
 Note that 
\begin{equation}
 \label{eq:est2}
 \|\hat \phi - w \|_{W^{2, \infty}(B_{r}(x))} \le \sqrt \epsilon (r^{2} + 2rd + 2d).
\end{equation}
 \item  Let 
 $$w_{1} (y) = \hat \phi (y) I_{\bar B_{r}(x)} (y) + (\epsilon + u^{g}(y)) 
 I_{\bar B_{r}^{c}(x)}(y),$$ 
 then $w_{1} \in USC$ due to $\hat \phi> u^{g}$ on $\partial B_{r}(x)$. Also, we have 
\begin{equation}
 \label{eq:est03}
 w_{1} = \hat \phi \ \hbox{ on } B_{r}; \quad 
 \|w_{1} - w\|_{L^{1}(\nu_{x}, B_{r}^{c}(x))} \le \epsilon \ \nu(B_{r}^{c}).
\end{equation}
 \item Next, $w_{2}$ is chosen from the continuous functions dominating 
 $w_{1}$ from its above, and 
 sufficiently close to $w_{1}$ in the following sense. Let $\mathcal C_{2}$ be 
$$\mathcal C_{2} = \{\bar w: \bar w - \epsilon \in C_{0} (\mathbb R^{d}); \ 
\bar w \ge w_{1} \hbox{ on } \mathbb R^{d}; \ 
\bar w = w_{1} \hbox{ on } \bar B_{r} (x)\}.$$
Since $w_{1}\in USC(\mathbb R^{d})$, $w_{1} (y) = g(y) + \epsilon$ 
for $y \notin O$, and $g\in C_{0}$, the set $\mathcal C_{2}$ is not empty.
If we let $\bar w$ run over all  such functions, 
then 
$\inf_{\bar w \in \mathcal C_{2}} (\bar w - w) (x) = 0$ 
for all $x\in B_{r}^{c}(x)$.
Then, we can apply  the monotone convergence theorem to have
 $$\inf_{\bar w \in \mathcal C_{2}} \| \bar w - w_{1}\|_{L^{1}(\nu_{x}, B_{r}^{c}(x))} = 0.$$
Therefore, we can take  $w_{2} \in \mathcal C_{2}$  
\begin{equation}
 \label{eq:est04}
 w_{2} = w_{1} \hbox{ on } B_{r}(x), \quad 
  \|w_{2} - w_{1}\|_{L^{1}(\nu_{x}, B_{r}^{c}(x))} \le \epsilon.
\end{equation}

 \item $w_{3} = \eta_{\epsilon'} * w_{2}$ is the convolution
  with a mollifier (see Appendix C.4 of \cite{Eva98}) 
  of radius $\epsilon' = \epsilon'(\epsilon)$, satisfying 
\begin{equation}
 \label{eq:est1}
 w_{3} \in C_{b}^{\infty} (\mathbb R^{d}); \ 
  \|w_{3} - w_{2}\|_{\infty} \le \frac 1 4 \epsilon; \ \hbox{ and }
  \|w_{3} - w_{2}\|_{W^{2,\infty}(B_{r/2}(x))} \le \sqrt \epsilon.
\end{equation}
Indeed, $w_{2} - \epsilon \in C_{0} (\mathbb R^{d})$ ensures that 
$$\hbox{ As } 
\epsilon' \to 0, 
\ 
w_{3} = \eta_{\epsilon'} * w_{2} =  \eta_{\epsilon'} *  (w_{2} - \epsilon) + \epsilon
\to w_{2} \ \hbox{ uniformly  on } \mathbb R^{d}.$$
Moreover, due to $w_{2} \in C^{\infty} (B_{r}(x))$,
 for any $\epsilon' < r/2$ and $y \in B_{r/2}(x)$, we have 
 $\partial_{x_{i}} w_{3} =  \eta_{\epsilon'} * \partial_{x_{i}} w_{2}$ and 
 $\partial_{x_{i}x_{j}} w_{3} =  \eta_{\epsilon'} * \partial_{x_{i}x_{j}} w_{2}$.
 This implies that
 $$\hbox{ As } 
\epsilon' \to 0, 
\ 
(Dw_{3}, D^{2}w_{3}) \to (Dw_{2}, D^{2}w_{2}),  \ \hbox{ uniformly  on } 
B_{r/2}(x).$$
This explains the existence of $\epsilon'$ satisfying \eqref{eq:est1}.
In addition, it also implies that
 \begin{equation}
 \label{eq:est05}
 \|w_{3} - w_{2}\|_{L^{1}(\nu_{x}, B_{r}^{c}(x))} 
 \le \frac 1 4 \epsilon \  \nu(B_{r}^{c}(x)).
\end{equation}
Moreover, we have, for any $y \in \mathbb R^{d}$
\begin{equation}
 \label{eq:01v5}
\begin{array}
 {ll}
  w_{3} (y) 
  & \ge w_{2}(y) - \frac 1 4 \epsilon \ge w_{1} (y) - \frac 1 4 \epsilon 
  \\ 
  & \displaystyle 
  \ge (\phi(y) + \sqrt \epsilon |y - x|^{2} - \frac 1 4 \epsilon) I_{\bar B_{r}(x)}(y) 
  + (\frac 3 4 \epsilon + u^{g}) I_{\bar B_{r}^{c}(x)}(y).
\end{array}
\end{equation}

\item Since $u^{g}$ is USC, there exists $x_{\epsilon}$ at which $u^{g} - w_{3}$ attains maximum over $\bar B_{r}(x)$. We denote
 $$x_{\epsilon} \in \arg\max_{\bar B_{r} (x) } (u^{g} - w_{3}), 
\hbox{ and }
\phi_{\epsilon} = w_{3} + (u^{g} - w_{3})(x_{\epsilon}).$$
We observe the following two useful estimations:
\begin{equation} \label{eq:est01}
(u_{g} - w_{3}) (x_{\epsilon})  \ge (u_{g} -w_{3}) (x) 
\ge (u_{g} -w_{2}) (x) - \frac 1 4 \epsilon
=  (u_{g} - \hat \phi) (x)  - \frac 1 4 \epsilon =  - \frac 1 4 \epsilon, 
\end{equation}
and 
\begin{equation}
 \label{eq:est02}
(u_{g} - w_{3}) (x_{\epsilon})  \le (u_{g} -w_{2}) (x_{\epsilon}) +  
\frac 1 4 \epsilon 
\le   (u_{g} - \hat \phi) (x_{\epsilon})  + \frac 1 4 \epsilon \le 
 -\sqrt \epsilon |x_{\epsilon} -x|^{2} 
+ \frac 1 4 \epsilon. 
\end{equation}
\end{enumerate}

Next, we shall verify that $\phi_{\epsilon}$ belongs to $J^{+}(u,x)$ and also
satisfies all conditions of Lemma \ref{p:02v6} as well. 
\begin{enumerate}
 \item $\phi_{\epsilon}$ is a constant shift of the smooth mollification $w_{3}$, and hence $\phi_{\epsilon} \in C^{\infty}(\mathbb R^{d})$ holds. Moreover, $\phi_{\epsilon} (x_{\epsilon}) = u_{g}(x_{\epsilon})$ is valid by its definition. In addition, we conclude 
 $\phi_{\epsilon} \in J^{+}(u, x_{\epsilon})$, since
\begin{itemize}
 \item  if
 $y\in \bar B_{r}(x)$, then 
 $(\phi_{\epsilon} -u_{g}) (y) = (u^{g} -w_{3})(x_{\epsilon}) - 
 (u^{g} - w_{3})(y) \ge 0$ since $x_{\epsilon}$ is maximum point of 
 $u^{g} -w_{3}$ on $B_{r}(x)$.
\item if
 $y\in B_{r}^{c}(x)$, then we have, by \eqref{eq:est01} and \eqref{eq:01v5}
 $$ 
\begin{array}
 {ll}
 (\phi_{\epsilon} -u_{g}) (y) &= (u^{g} -w_{3})(x_{\epsilon}) +
 ( - u^{g} + w_{3})(y) 
 \\&
 \ge (u^{g} -w_{3})(x_{\epsilon}) + \frac 3 4 \epsilon 
 \\&
 \ge \frac 1 2 \epsilon >0.
\end{array}
 $$
\end{itemize}
\item From \eqref{eq:est01} and \eqref{eq:est02}, we immediately write
$ -\sqrt \epsilon |x_{\epsilon} -x|^{2} 
+ \frac 1 4 \epsilon \le - \frac 1 4 \epsilon$ or equivalently
$
 |x_{\epsilon} - x|^{2} \le \frac 1 2 \sqrt \epsilon.$ This implies $\lim_{\epsilon \to 0} x_{\epsilon} = x$.
 \item If $y\in B_{r}(x)$, then \eqref{eq:est01} and \eqref{eq:est02} again implies that $\phi_{\epsilon}$ is a constant shift from $w_{3}$ with
 $$|\phi_{\epsilon}(y) - w_{3}(y)| < \frac 1 4 \epsilon.$$
 Together with \eqref{eq:est2}, \eqref{eq:est03}, \eqref{eq:est04}, and 
 \eqref{eq:est1}, we obtain
 $$\|\phi_{\epsilon} - w\|_{W^{2,\infty}(B_{r/2}(x))} \le \sqrt \epsilon (r^{2} + 
 2rd + 2d + 1) + \frac 1 4 \epsilon \to 0, \hbox{ as } \epsilon \to 0.$$
 \item Finally, we shall check 
 $\|\phi_{\epsilon} - w\|_{L^{1}(\nu_{x}, B_{r}^{c}(x))}  \to 0$.  First,
 we write from definition of $\phi_{\epsilon}$ that
 $$\|\phi_{\epsilon} - w\|_{L^{1}(\nu_{x}, B_{r}^{c}(x))} 
 \le \|w_{3} - w \|_{L^{1}(\nu_{x}, B_{r}^{c}(x))}  + 
 | (u^{g} - w_{3})(x_{\epsilon})| 
 \cdot \nu(B_{r}^{c}).
$$
The first term $\|w_{3} - w \|_{L^{1}(\nu_{x}, B_{r}^{c}(x))} \to 0$ holds 
due to \eqref{eq:est03}, \eqref{eq:est04}, and \eqref{eq:est05}.
The second term $| (u^{g} - w_{3})(x_{\epsilon})| 
 \cdot \nu(B_{r}^{c})\to 0$ holds due to \eqref{eq:est01} and \eqref{eq:est02}.
\end{enumerate}
We finish the proof by applying Proposition \ref{p:01v5}.
\end{proof}

\section{A proof of Perron's method} \label{sec:per}
In this section, we prove Lemma~\ref{l:perron1}, and
Proposition~\ref{p:pm} is the direct consequence of Lemma~\ref{l:perron1}.

\begin{proposition}
 \label{p:perron1}
 If $u$ and $v$ are both 
 subsolutions of \eqref{eq:pde11} - \eqref{eq:bd11},
 then the new function $\max\{u, v\}$ is also a subsolution 
 of \eqref{eq:pde11} - \eqref{eq:bd11}.
\end{proposition}

The proof of Proposition  \ref{p:perron1} is referred to Theorem 2 of \cite{BI08}. Next, 
Proposition \ref{p:cp} and \ref{p:perron1} enables us to follow 
the same {\it bump construction} to as of Lemma 4.4 of \cite{CIL92},
which eventually
leads to Perron's method via Lemma \ref{l:perron1} in this below.

\begin{lemma}
 \label{l:perron1}
 Let $u$ be a subsolution of \eqref{eq:pde11} - \eqref{eq:bd11}, 
 and $u_{*}$ fail to be a supersolution at some $\hat x \in O$.
 Then, for any small enough $\kappa >0$, there exists a subsolution 
 $u_{\kappa}$ such that
 $$u_{\kappa}  \ge u(x); \ \sup_{O} (u_{\kappa} - u) >0; \  \hbox{ and }
 u_{\kappa}  = u \hbox{ on } B_{\kappa}(\hat x).$$
\end{lemma}

\begin{proof}
 For simplicity $\hat x = 0$ and there exists $\phi \in J^{-} (u_{*}, 0)$ such that 
 $$\hat F(\phi, 0) := F(\phi, 0) + \phi(0) - \ell(0) = - \epsilon <0.$$
 Since $\hat F(\phi, \cdot)$ is continuous, there exists $\kappa_{0}>0$ such that 
 $$\sup_{x \in B_{\kappa_{0}}} \hat F(\phi, x) < - \frac{\epsilon}{2}.$$ 
 We fix arbitrary $\kappa <\kappa_{0}$. Let $u_{\gamma}$ be a function of
 $$
 u_{\gamma} (x) = \phi(x) + 
 \gamma(\kappa^{2} - |x|^{2}) I_{B_{2\kappa}}(x) := \phi(x) + \psi_{\kappa}(x).$$
 If $x\in B_{\kappa}$, then we have
\begin{enumerate}
 \item  $$H(u_{\gamma}, x, a) = 
 H(u_{\gamma}, x, a)  - \gamma(tr(A(a)) + b(a) \cdot x)
 \ge 
 H(u_{\gamma}, x, a)  - \gamma c_{\kappa, 1},$$
 where $c_{\kappa, 1}$ is a number defined by
 $c_{\kappa, 1} :=  \sup_{x\in B_{\kappa}, a\in [\underline a, \overline a]}
 | tr(A(a)) + b(a) \cdot x | <\infty$. This means
 $$- \inf_{a\in [\underline a, \overline a]} H(u_{\gamma}, x, a) 
 \le - \inf_{a\in [\underline a, \overline a]} H(\phi, x, a)  + \gamma c_{\kappa, 1}.$$
\item  On the other hand, we also have
 $$- \mathcal I(u_{\gamma}, x) 
 = -  \mathcal I(\phi, x) + \gamma \mathcal I(\psi_{\kappa}, x)
 \le -  \mathcal I(\phi, x) + \gamma c_{\kappa, 2},$$
 where 
 $c_{\kappa, 2} := \sup_{x\in B_{\kappa}} |\mathcal I(\psi_{k}, x) |
 < \infty$ holds due to the continuity of $\mathcal I(\psi_{\kappa}, \cdot)$,
 see Proposition \ref{p:02v5}.
\end{enumerate}
 Therefore, we conclude that, with $c_{\kappa} := c_{\kappa, 1} + c_{\kappa, 2}$
 $$\hat F(u_{\gamma}, x) 
 \le F(\phi, x) + \gamma c_{\kappa} + \phi(x) - \ell(x)
 = \hat F(\phi, x) + \gamma c_{\kappa}.$$
 Now we take $\gamma = \frac{\epsilon}{2 c_{\kappa}}$ and we have $u_{\gamma}$ be a subsolution on $B_{\kappa}$. Then,  we have
\begin{enumerate}
 \item  if $
 x\in B_{\kappa},$
 then 
 $$u_{\gamma}(x)  = \phi(x) + \gamma (\kappa^{2} - |x|^{2}) I_{B_{2\kappa}} (x) \le \phi(x) \le u_{*}(x) \le u(x),$$
 \item and $u_{\gamma}(0) = \phi(0) + \gamma \kappa^{2} >\phi(0) = u_{*}(0)$ implies that there exists $x_{n} \to 0$ such that
 $u_{\gamma}(x_{n}) > u(x_{n})$.
\end{enumerate}
Finally, we take $u_{\kappa} = \max\{u_{\gamma}, u\}$ to finish the proof
by Proposition \ref{p:perron1}.
\end{proof}

\section{Skorohod metric in c\`adl\`ag space} \label{sec:skorohod}
We denote by $\mathbb D^{d}_{t}$ the collection of c\`adl\`ag functions on $[0, t)$ taking values in $\mathbb R^{d}$. In particular, $\mathbb D^{d}_{\infty}$ is the collection of  c\`adl\`ag functions on $[0, \infty)$.
According to \cite{Bil99}, one can impose Skorohod metric $d^{o}_{t}$ in the space $\mathbb D^{d}_{t}$ as of below to make the space complete. It is proven in \cite{Bil99} that, 
$\mathbb D_{t}^{d}$ (resp. $\mathbb D_{\infty}^{d}$) is complete under the metric $d_{t}^{o}$ (resp. $d_{\infty}^{o}$), which is equivalent to J1 Skorohod metric.
\begin{enumerate}
 \item For  $t\in [0, \infty)$, we define 
 the sup norm 
\begin{equation}
 \label{eq:supnorm}
\|x\| = \sup_{0\le s < t} |x(t)|.
\end{equation}

 \item For  $t\in [0, \infty)$, we
  denote by $\Lambda_{t}$
 by the class of strictly increasing continuous mappings of $[0, t]$ onto
 itself. In particular, $\lambda(0) = 0$ and $\lambda(t) = t$ for all
 $\lambda \in \Lambda$. The identity $I$ on $[0,t]$ also 
 belongs to $\Lambda_{t}$. We can define a functional in 
 $\Lambda_{t}$ by
 $$\|\lambda\|^{o} = \sup_{0\le s< r\le t} \Big | \log \frac{\lambda\circ r - 
 \lambda \circ s}{r - s} \Big|, \ \forall \lambda \in \Lambda_{t}.$$
 Note that $\|\lambda\|^{o}$ may not be necessarily finite in $\Lambda_{t}$.
 \item  For  $t\in [0, \infty)$, define the distance function $d_{t}^{o}(x,y)$ in $\mathbb D^{d}_{t}$ by
 $$d_{t}^{o}(x,y) = \inf_{\lambda\in \Lambda_{t}}
 \{\|\lambda \|^{o} 
 \vee \|x - y \circ\lambda\|\},
 \ \forall x,y \in \mathbb D^{d}_{t}.$$
 \item We define the distance function $d_{\infty}^{o}(x,y)$ in
  $\mathbb D^{d}_{\infty}$ by
  $$d^{o}_{\infty} (x,y) = \sum_{m=1}^{\infty} 2^{-m} (1 \wedge d^{o}_{m}(x^{m}, y^{m}))
   \ \forall x,y \in \mathbb D^{d}_{\infty},$$
   where $x^{m}(t) = g_{m}(t) x(t)$ for all $t\ge 0$ with a continuous function
   $g_{m}$ given by
   $$g_{m}(t) = \left\{ 
\begin{array}
 {ll}
 1, & \hbox{ if } t\le m-1, \\
 m-t, & \hbox{ if } m-1 \le t \le m, \\
 0, & \hbox{ otherwise. }
\end{array}
\right.
$$
\end{enumerate}

Define a projector $\Pi: \mathbb D^{d}_{\infty} \times [0, \infty) \mapsto \mathbb R^{d}$ by
\begin{equation}
 \label{eq:pi02}
 \Pi(\omega, t) = \omega(t).
\end{equation}

\begin{proposition}
 \label{p:pi01}
 $\omega \mapsto \Pi(\omega, t)$ is continuous at $\omega_{0}$ 
 if $t\mapsto \omega_{0}(t)$ is continuous at $t$.
\end{proposition}
\begin{proof}
 It's a consequence of Theorem 12.5 of \cite{Bil99}.
\end{proof}

Finally, we give two useful examples.

\begin{example}\label{e:01} 
 For simplicity, consider $O = (0, 1) \subset \mathbb R$. 
 \begin{itemize}
 \item $T_{O^{c}}$ is not upper semicontinuous at $\omega$ given by $$\omega(t) = |t - 1/ 2 |,$$
 which is illustrated in Figure~\ref{fig:1}
 since $ \lim_{n}T_{O^{c}}(\omega_{n}) = 3/2 > 1/2 = T_{O^{c}}(\omega)$ where $\omega_{n} = \omega+ 1/n$.
 \item $T_{O^{c}}$  is not   lower semicontinuous  at $\omega$ given by
 $$\omega(t) = (-t +  1 /3) I(t < 1 /3) + (-t +  2/ 3) I(t \ge 1/ 3),$$
 which is illustrated in Figure~\ref{fig:2}.
 In fact, setting $\omega_{n} = \omega - 1/n$, we have
 $\lim_{n} T_{O^{c}}(\omega_{n}) = 1/3 < 2/3 = T_{O^{c}}(\omega).$ 
 \end{itemize}
 
 \begin{figure}
\centering
\begin{minipage}{.5\textwidth}
  \centering
  \includegraphics[width=.7\linewidth]{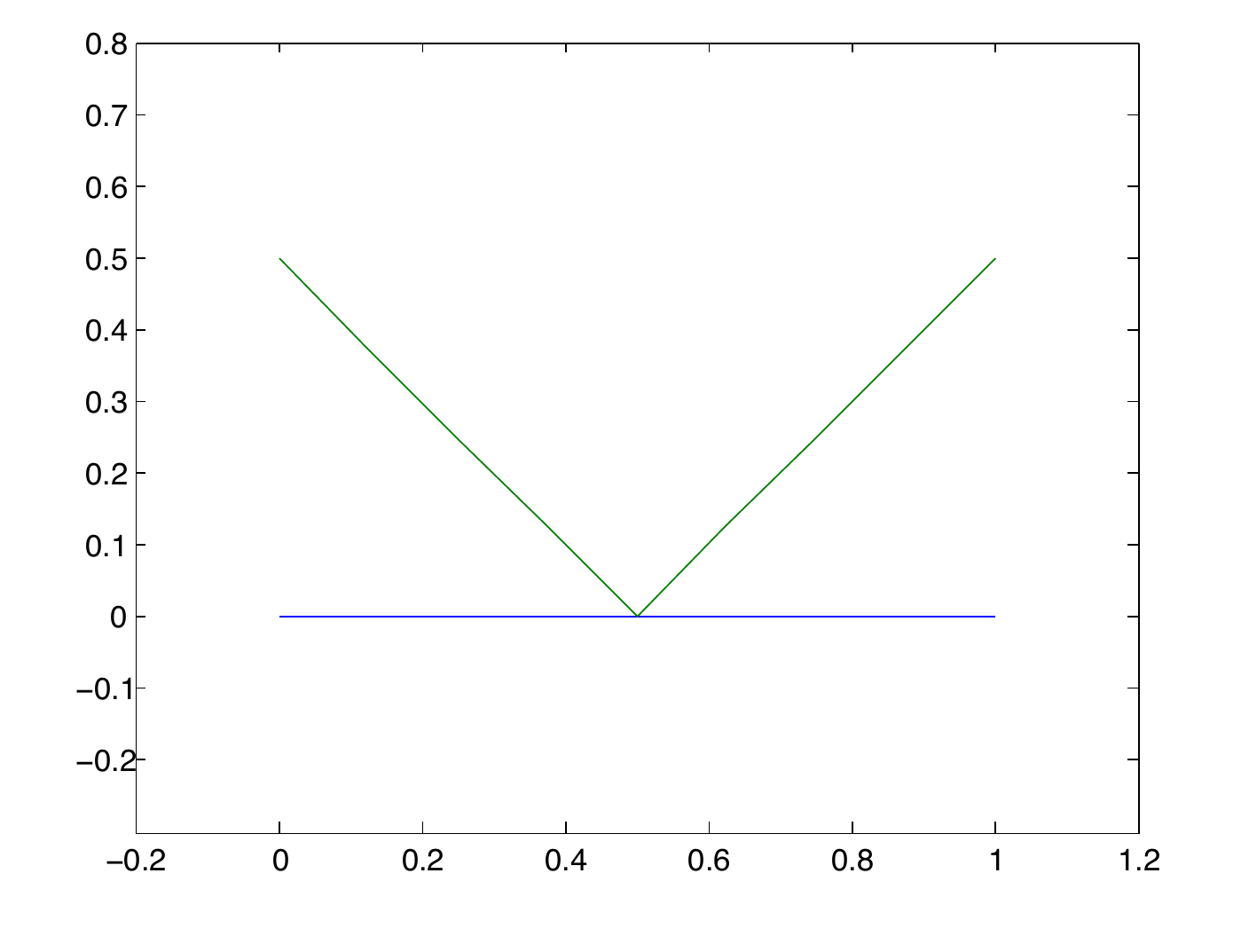}
  \caption{Shift up}
  \label{fig:1}
\end{minipage}%
\begin{minipage}{.5\textwidth}
  \centering
  \includegraphics[width=.8\linewidth]{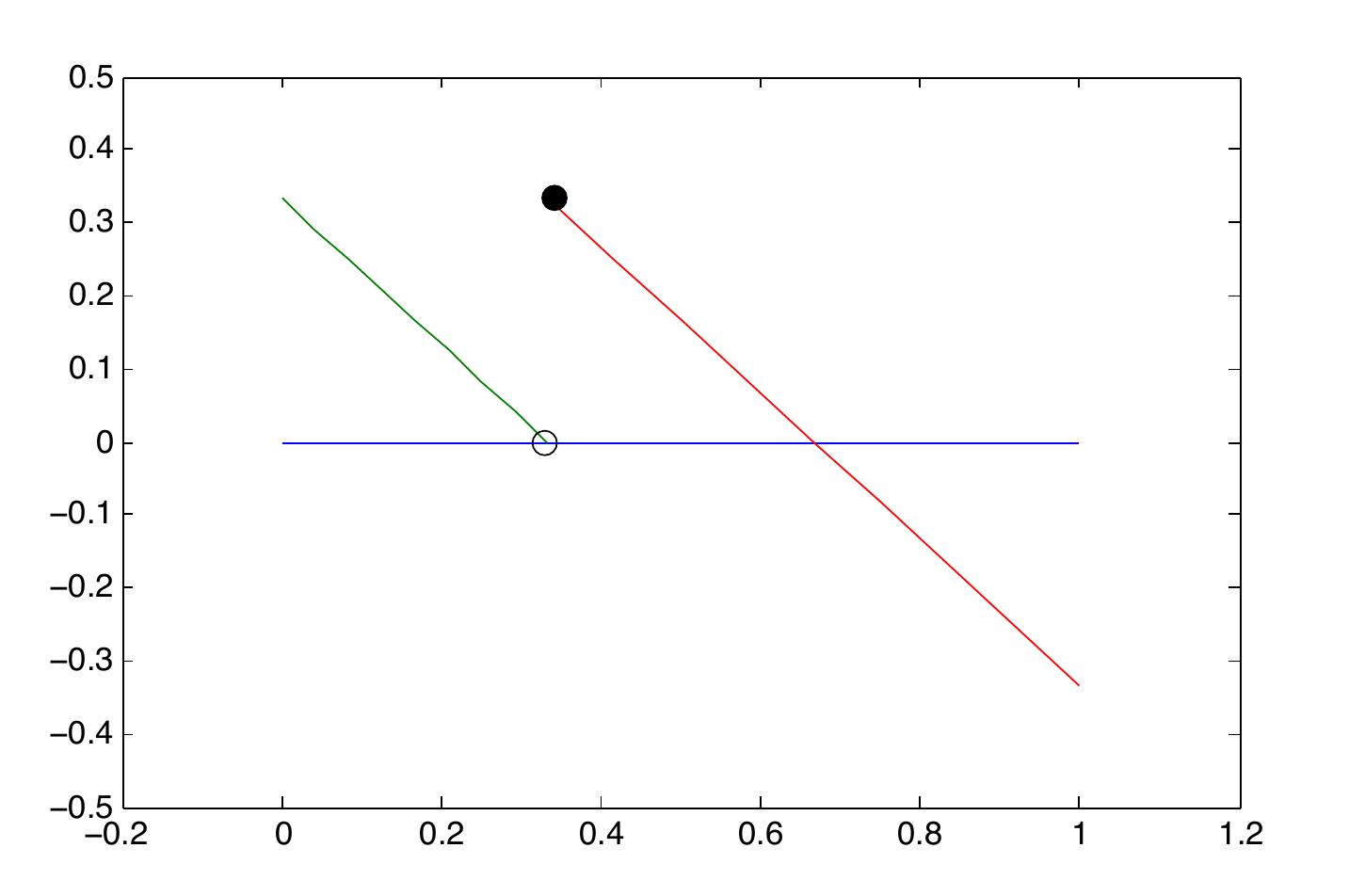}
  \caption{Shift down}
  \label{fig:2}
\end{minipage}
\end{figure}

\end{example}

\begin{example}
 \label{e:02v5}
 Let  $O = (0, 1)$ and
 $$\omega (t) = 1 - t - I(t\ge 1),$$
 which is illustrated in Figure~\ref{fig:3}.
 Since $\omega \in \Gamma_{O}$, 
 we have the continuity of $T_{O^{c}}$ at $\omega$ 
 by Theorem~\ref{t:hit01}.
 If we take $\omega_{n} = \omega - 1/n$ for all $n\in \mathbb N$, we have
 $\omega_{n} \to \omega$ in uniform topology, hence in Skorohod topology.
 Therefore, $T_{O^{c}}(\omega_{n})= 1- 1/n \to 1 = T_{O^{c}}(\omega)$, which
 supports Theorem~\ref{t:hit01}. However, we have 
 $$\Pi_{O}(\omega_{n}) = 0 \not\to -1 = \Pi_{O}(\omega).$$
 
  \begin{figure}
\centering
\includegraphics[scale = .5]{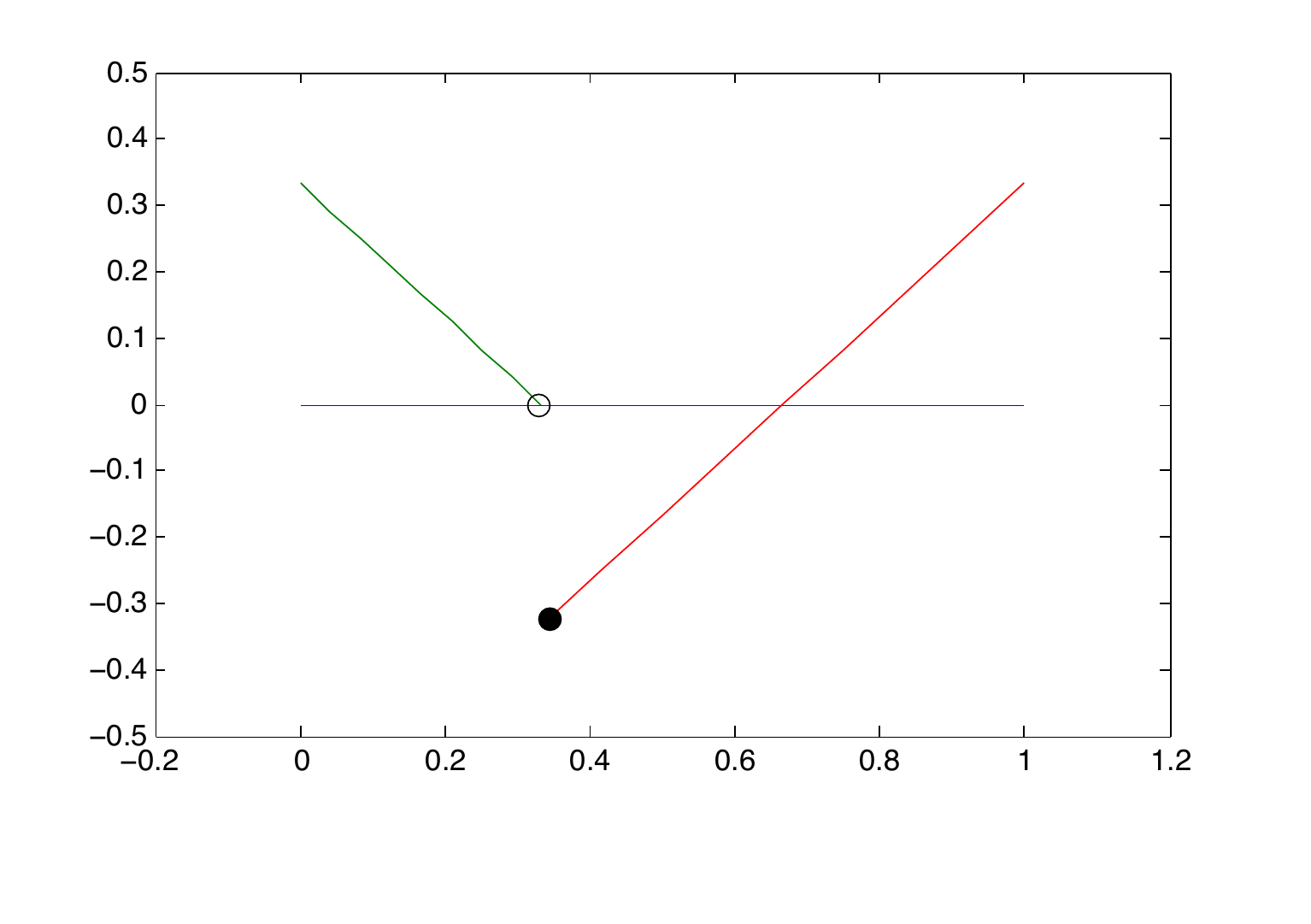} 
\caption{A small down shift makes a big change in the state at the 
first exit time}\label{fig:3}
\end{figure}

\end{example}

\section*{Acknowledgments}
Q. Song is grateful to Guy Barles and Pierre-Louis Lions for helpful comments.

\bibliographystyle{siamplain}
\bibliography{refs}
\end{document}